\documentclass{amsart}

\usepackage{graphicx}
\usepackage{amsmath,amssymb,amsfonts}
\usepackage{mathtools}
\usepackage{bm}
\usepackage{mathrsfs}
\usepackage[numbers]{natbib}
\usepackage{xcolor}
\usepackage[hidelinks]{hyperref}

\newtheorem{theorem}{Theorem}[section]
\newtheorem{lemma}[theorem]{Lemma}
\newtheorem{proposition}[theorem]{Proposition}
\newtheorem{corollary}[theorem]{Corollary}
\theoremstyle{definition}
\newtheorem{definition}[theorem]{Definition}
\newtheorem{remark}[theorem]{Remark}
\numberwithin{equation}{section}

\allowdisplaybreaks

\newcommand{\N}{\mathbb N}
\newcommand{\Sph}{\mathbb S}
\newcommand{\supp}{\operatorname{supp}}

\newcommand{\dd}{\,\mathrm d}

\newcommand{\Pmax}{\mathcal P}
\newcommand{\MHL}{\mathcal M}
\newcommand{\sigmaK}{\sigma_k}

\title[Sharp Besov--variation embeddings]
{Sharp embeddings between quasi-Banach Besov spaces and shallow ReLU variation spaces}
\author[Y. Li]{Yuwen Li}
\address{School of Mathematical Sciences, Zhejiang University, 866 Yuhangtang Road,
Hangzhou, Zhejiang 310058, People's Republic of China}
\email{liyuwen@zju.edu.cn}
\author[Y. Wang]{Yupeng Wang}
\address{School of Mathematical Sciences, Zhejiang University, 866 Yuhangtang Road,
Hangzhou, Zhejiang 310058, People's Republic of China}
\email{yupengw@zju.edu.cn}
\subjclass[2020]{Primary 46E35; Secondary 41A46, 42B25}
\keywords{Besov space, variation space, shallow neural network,
rectified power unit, embedding theorem, quasi-Banach space}
\date{}

\begin{document}

\begin{abstract}
Let $\mathbb{D}$ be the ridge dictionary generated by
$\operatorname{ReLU}^k$ on a bounded Lipschitz domain
$\Omega\subset\mathbb{R}^d$. We establish sharp embeddings between
isotropic Besov spaces and the associated variation space
$\mathcal L_1(\mathbb{D})$ in the quasi-Banach range $0<p\leq 1$.
Specifically,
\[
B^s_{p,q}(\Omega)\hookrightarrow \mathcal L_1(\mathbb{D})
\]
when $s\geq k+d/p$ for $0<q\leq 1$, and when $s>k+d/p$ for
$1<q\le\infty$. A rescaled-bump construction shows that this smoothness
threshold is sharp. Conversely, for $0<p<1$,
\[
\mathcal L_1(\mathbb{D})\hookrightarrow B^{k+1}_{p,2}(\Omega),
\]
and both the smoothness $k+1$ and the fine index $2$ are optimal.
The forward embedding converts Besov regularity of functions and solutions
of partial differential equations into quantitative finite-width
approximation bounds and greedy-algorithm convergence guarantees for
shallow $\operatorname{ReLU}^k$ networks. The proofs use
Littlewood--Paley localization, Fourier--Radon representations,
measure-valued derivatives, and vector-valued singular integrals.
\end{abstract}

\maketitle

\section{Introduction}\label{sec:introduction}
Neural-network-based PDE solvers currently constitute an important class of nonlinear numerical 
methods for partial differential equations (PDEs). These methods replace a
prescribed mesh or linear trial space by a trainable network class and are
therefore attractive for high-dimensional problems. Their best achievable
accuracy, however, is governed by how efficiently the underlying PDE
solution can be represented by networks of controlled width and
coefficient size. It is thus natural to seek a functional-analytic
benchmark for shallow-network approximation.

We fix integers $d\geq1$ and $k\in\mathbb N_+$, a bounded Lipschitz domain
$\Omega\subset\mathbb{R}^d$, set
$\sigma_k(t):=\max\left\{t,0\right\}^k$, and choose $c>0$ sufficiently
large. Here and below,
$\mathbb S^{d-1}:=\{\omega\in\mathbb{R}^d:|\omega|=1\}$ denotes the unit
$(d-1)$-sphere. The normalized ridge dictionary is
\[
 \mathbb{D}:=\{\sigma_k(\omega\cdot{}-b):
 (\omega,b)\in\mathbb S^{d-1}\times[-c,c]\}.
\]
For a function $f$ on $\Omega$, the associated variation norm is
\[
 \|f\|_{\mathcal L_1(\mathbb{D})}
 :=\inf\left\{\|\mu\|_{\mathrm{TV}}:
 f(x)=\int_{\mathbb S^{d-1}\times[-c,c]}
 \sigma_k(\omega\cdot x-b)\,\dd\mu(\omega,b)\right\}.
\]
The infimum is taken over all finite signed Borel measures $\mu$ on
$\mathbb S^{d-1}\times[-c,c]$ that represent $f$ on $\Omega$.
Thus $\mathcal L_1(\mathbb{D})$ is an infinite-width, coefficient-$\ell^1$
model of shallow $\operatorname{ReLU}^k$ networks. It is the natural
atomic space associated with $\mathbb{D}$ and connects Fourier and convex
descriptions of network complexity with Radon-transform and ridge-spline
formulations
\cite{Barron1993,Bach2017,Savarese2019,Ongie2020,ParhiNowak2021}.
With the above normalization, $\mathcal L_1(\mathbb{D})$ is also a Barron-type
variation space: the total variation of the representing measure is the
infinite-width analogue of an outer $\ell^1$ or path-norm budget, up to
conventions concerning parameter normalization and polynomial terms.
Related Barron-space approximation, representation, and embedding results
can be found in
\cite{EMaWu2022,EWojtowytsch2022,WuBarronEmbedding2023}.

The regularity scale relevant to nonlinear approximation is supplied by
Besov spaces. If $\{\Delta_j\}_{j\geq0}$ is an inhomogeneous
Littlewood--Paley resolution, then, up to equivalence of quasi-norms
\cite{Triebel,FrazierJawerth1990},
\[
 \|f\|_{B_{p,q}^s(\mathbb{R}^d)}
 =\left(\sum_{j=0}^\infty
2^{jsq}\|\Delta_jf\|^q_{L^p(\mathbb{R}^d)}
 \right)^{1/q},
\]
with the usual modification for $q=\infty$. On $\Omega$, we use the
restriction-space quasi-norm. Sobolev regularity records global
differentiability, whereas Besov regularity also captures localized and
multiscale behavior. Results of Dahlke and DeVore show that solutions of
elliptic boundary value problems on Lipschitz and polygonal domains may
possess substantially stronger regularity in approximation-relevant
Besov scales than in the corresponding Sobolev scales
\cite{DahlkeDeVore,DahlkePolygonal}. This suggests that the neural
representability of an elliptic solution should be related to its Besov
regularity.

Diagonal Besov spaces $B^s_{p,p}(\Omega)$ with $p<1$ also play an
important role in adaptive finite element approximation. Their
quasi-norms capture the spatial sparsity of corner and edge
singularities, which is precisely the structure exploited by locally
refined meshes. Besov regularity is therefore closely connected with
best adaptive approximation rates and with the approximation classes
that govern quasi-optimal adaptive methods
\cite{BDDP,GaspozMorin2009,Gantumur2017}. This connection explains why
the quasi-Banach range $p<1$ arises naturally in PDE approximation. 

The variation norm is directly relevant to finite networks as well. Let
\[
 \Sigma_n(\mathbb{D}):=\left\{\sum_{i=1}^n a_i v_i:
 a_i\in\mathbb{R},\ v_i\in\mathbb{D}\right\}.
\]
For every $2\leq r<\infty$, \cite[Theorem~3]{SiegelXuSharp} gives
\begin{equation}\label{eq:SiegelXu}
 \inf_{g_n\in\Sigma_n(\mathbb{D})}\|f-g_n\|_{L^r(\Omega)}
 \lesssim
 n^{-1/2-(k+1/r)/d}\|f\|_{\mathcal{L}_1(\mathbb{D})}.
\end{equation}
Since the unit ball of $\mathcal L_1(\mathbb{D})$ is the closed absolutely convex
hull of $\mathbb{D}$, it is the natural target class for dictionary-greedy
algorithms. The pure greedy algorithm, also known as matching pursuit,
selects an atom of maximal residual correlation. The relaxed greedy
algorithm updates by a relaxed convex combination, whereas the
orthogonal greedy algorithm recomputes the best approximation in the
span of all selected atoms. For a uniformly bounded dictionary in a
Hilbert space $H$, the relaxed and orthogonal algorithms explicitly construct
$f_n\in\Sigma_n(\mathbb{D})$ satisfying
\cite{DeVoreTemlyakov1996,BarronCohenDahmenDeVore2008}
\begin{equation*}
 \|f-f_n\|_H\lesssim n^{-1/2}\|f\|_{\mathcal L_1(\mathbb{D})}.
\end{equation*}
For the compact dictionary $\mathbb{D}$, the orthogonal greedy algorithm satisfies the improved entropy-based error bound \cite{SiegelXuOrthogonalGreedy,LiSiegel2024Greedy}:
\begin{equation}\label{eq:OGA}
 \|f-f_n\|_H\lesssim \varepsilon_n(\mathcal{K}(\mathbb{D}))_H\|f\|_{\mathcal L_1(\mathbb{D})},
\end{equation}
where $\varepsilon_n(\mathcal{K}(\mathbb{D}))_H$ denotes the $n$th metric
entropy number  of the closed absolutely convex hull of $\mathbb{D}$.
Such an estimate was extended to the Chebyshev greedy algorithm in
uniformly smooth Banach spaces 
\cite{Li2025Chebyshev}. The space
$\mathcal L_1(\mathbb{D})$ is also used in
the convergence analysis of greedy algorithms for rational
approximation~\cite{LiLi2025rEIM,LiZikatanovZuo2024}.

The preceding considerations lead to a general question: to what extent
does isotropic Besov regularity control the variation norm of a shallow
ridge representation, and what Besov regularity is forced in the reverse
direction? Direct
finite-width approximation of classical Sobolev and Besov classes has
been studied for deep $\operatorname{ReLU}$ networks and shallow
$\operatorname{ReLU}^k$ networks
\cite{SiegelDeep2023,MaoSiegelXu2024}. Related high-order approximation
results for spectral Barron classes and sharp rates for shallow variation
spaces appear in \cite{SiegelXuHighOrder,SiegelXuSharp}.
Li, Liu, and Shi~\cite{LiLiuShi2026} analyze $L^r$-density ridge-integral
spaces for $1\leq r\leq2$ and derive shallow-network approximation bounds
and Sobolev-space estimates through spectral Barron embeddings. Going
beyond finite-width approximation estimates, we establish a continuous
embedding into the infinite-width atomic space $\mathcal L_1(\mathbb{D})$.

On the representation-space side, standard shallow variation spaces have
been characterized in~\cite{SiegelXuVariation}, and closely related
Radon-transform and ridge-spline descriptions appear in
\cite{Ongie2020,ParhiNowak2021,Unser2023}. He and Tian~\cite{HeTian2026}
recently introduced Radon-domain $L^r$ ridge-integral spaces. They identify
the case $r=2$ with the critical Sobolev space
$H^{k+(d+1)/2}$, establish a sharp Sobolev sandwich for
$1<r<\infty$, and derive sampling-based approximation rates. These
coefficient-density spaces, Fourier-defined neural
spaces~\cite{LiaoMing2025}, intrinsic weighted variation spaces
\cite{DeVoreWeighted2025}, and generalized Barron spaces with weighted
parameter measures~\cite{LuZhang2026} provide complementary
representation norms. We resolve the sharp embedding problem for the
standard measure-valued $\operatorname{ReLU}^k$ variation space in the
quasi-Banach Besov range $p<1$.

Our main interest is the forward direction. For a given function, and
especially for a solution of a PDE, Besov regularity is often known or can
be derived from established regularity theory, whereas membership in
$\mathcal L_1(\mathbb{D})$ and quantitative control of the variation norm are
typically much less understood. The forward embedding therefore converts
available analytic regularity into a controlled shallow-network
representation and finite-width approximation guarantees. The converse
embedding identifies the optimal Besov smoothness and fine index forced
by this variation norm.

To the best of our knowledge, these are the first sharp embeddings
between quasi-Banach Besov spaces and the standard measure-valued shallow
$\operatorname{ReLU}^k$ variation space. Our first result gives the
critical Besov-to-variation embedding and its above-critical extension.

\begin{theorem}\label{Thm:BesovEmbeddingVariation}
Let $\Omega\subset\mathbb{R}^d$ be a bounded Lipschitz domain and let
$0<p\leq1$. Assume that
\begin{equation*}
 s>k+\frac{d}{p}\text{ and }0<q\leq\infty,
 \qquad\text{or}\qquad
 s=k+\frac{d}{p}\ \text{and}\ 0<q\leq1.
\end{equation*}
Then $B_{p,q}^s(\Omega)\hookrightarrow\mathcal L_1(\mathbb{D})$.
\end{theorem}

Combining the embedding in Theorem~\ref{Thm:BesovEmbeddingVariation} with 
\eqref{eq:SiegelXu} gives, for every
$f\in B_{p,q}^s(\Omega)$ in the range of the theorem and every
$2\leq r<\infty$,
\[
 \inf_{g_n\in\Sigma_n(\mathbb{D})}\|f-g_n\|_{L^r(\Omega)}
 \lesssim
 n^{-1/2-(k+1/r)/d}\|f\|_{B_{p,q}^s(\Omega)}.
\]
Similarly, the entropy-based convergence rate \eqref{eq:OGA} of the
orthogonal greedy algorithm translates into
\begin{equation*}
 \|f-f_n\|_{L^2(\Omega)}\lesssim
 \varepsilon_n(\mathcal K(\mathbb{D}))_{L^2(\Omega)}\|f\|_{B^s_{p,q}(\Omega)}.
\end{equation*}
We also remark that Theorem \ref{Thm:BesovEmbeddingVariation} can be applied to the entropy-based error bound for Chebyshev greedy algorithms in Banach spaces in \cite{Li2025Chebyshev}.

Another concrete application of Theorem \ref{Thm:BesovEmbeddingVariation},
now for the ReLU dictionary with $k=1$, is the Poisson boundary value
problem on the unit square $\Omega=(0,1)^2$. Let $u\in H_0^1(\Omega)$ solve
\begin{equation*}
 -\Delta u=1\quad\text{in }\Omega,
 \qquad u=0\quad\text{on }\partial\Omega.
\end{equation*}
The solution is smooth away from the vertices while its standard
right-angle corner expansion contains a singular term of the form
$r^2\log(r)\sin(2\theta)$ in local polar coordinates centered at a
vertex. The corresponding Besov estimates for polygonal corner
singularities give
$u\in B^s_{p,p}(\Omega)$ for $0<p\leq1$ and $s<2+2/p$
\cite{DahlkePolygonal}. Therefore, Theorem~\ref{Thm:BesovEmbeddingVariation} with $k=1$ applies at its
critical endpoint and yields $u\in\mathcal L_1(\mathbb{D})$. Combining this
membership with \eqref{eq:SiegelXu}, we obtain, for $2\leq r<\infty$,
\begin{equation*}
 \inf_{g_n\in\Sigma_n(\mathbb{D})}\|u-g_n\|_{L^r(\Omega)}
 \lesssim n^{-1-1/(2r)}\|u\|_{B^{1+2/p}_{p,p}(\Omega)}.
\end{equation*}

The norm for measuring the preceding approximation errors can be replaced by any $W^{t,r}$ Sobolev norm, provided $\mathbb{D}$ is compact in $W^{t,r}(\Omega)$ and the approximation rate is properly adjusted, see Section \ref{sect:prelim} for details.

\begin{remark}
A rescaled-bump construction shows that the smoothness threshold $s$ in
Theorem~\ref{Thm:BesovEmbeddingVariation} cannot be lowered. In contrast,
Mao, Siegel, and Xu~\cite[Theorem~1]{MaoSiegelXu2024} prove
$B_{2,2}^{k+(d+1)/2}(\Omega)\hookrightarrow\mathcal L_1(\mathbb{D})$ when $\Omega$
is the unit ball. At $p=2$, the formal scaling $s=k+d/p$ gives
$s=k+d/2$, whereas their sufficient index is $k+(d+1)/2$. Thus their
condition lies one half-order above the sharp scaling line $s=k+d/2$. Theorem
\ref{Thm:BesovEmbeddingVariation} attains $s=k+d/p$ throughout the
quasi-Banach range $0<p\leq1$, including the endpoint when $q\leq1$.
Although we state the theorem only in this quasi-Banach range, the
techniques developed here are expected to extend to every $p>1$, we leave
that extension to future work.

For $1<p<2$, He and Tian~\cite[Theorem~3.3]{HeTian2026} prove
\[
 W^{k+d/p+1/p',p}(\Omega)\hookrightarrow\mathcal{RL}_k^p(\Omega)\hookrightarrow\mathcal L_1(\mathbb{D}),
 \qquad \frac1{p'}=1-\frac1p,
\]
where $\mathcal{RL}_k^p(\Omega)$ is called the Radon-domain space. Therefore, this embedding also requires an additional smoothness order $1/p'$  in view of our sharp scaling relation $s=k+d/p$. Theorem~\ref{Thm:BesovEmbeddingVariation}
attains the scaling line $s=k+d/p$ for $0<p\leq1$, with $q\leq1$ at
equality. The sharpness in \cite[Proposition~3.2]{HeTian2026} concerns the
embedding into $\mathcal{RL}_k^p(\Omega)$. Our rescaled-bump construction
establishes sharpness directly for the larger measure-valued space
$\mathcal L_1(\mathbb{D})$.
\end{remark}

Our second result identifies the optimal converse embedding.

\begin{theorem}\label{thm:variation-to-besov}
Let $\Omega\subset\mathbb{R}^d$ be a bounded Lipschitz domain and let
$0<p<1$. Then we have 
\begin{equation*}
 \mathcal L_1(\mathbb{D})\hookrightarrow B_{p,2}^{k+1}(\Omega).
\end{equation*}
For every $0<q<2$,
\begin{equation*}
 \mathcal L_1(\mathbb{D})\not\hookrightarrow B_{p,q}^{k+1}(\Omega),
\end{equation*}
and for every $s>k+1$ and $0<q\leq\infty$,
\begin{equation*}
 \mathcal L_1(\mathbb{D})\not\hookrightarrow B_{p,q}^s(\Omega).
\end{equation*}
\end{theorem}

The forward proof combines a spatial atomic decomposition of
Littlewood--Paley blocks with a Fourier--Radon representation and a
one-dimensional Peano formula. The spatial coefficient sum produces the critical factor
$2^{jd/p}$, while the scale sum yields the distinction between the
critical and above-critical cases. The converse proof uses measure-valued
derivatives and a vector-valued Calder\'on--Zygmund estimate. Separated
and lacunary ridge constructions prove the sharpness assertions in the
second theorem.

The rest of the paper is organized as follows. Section~2 introduces the
variation and Besov spaces. Section~3 develops the localized ridge
representation used in the forward embedding. Section~4 proves
Theorem~\ref{Thm:BesovEmbeddingVariation} and the sharpness of its
smoothness threshold. Section~5 proves Theorem~\ref{thm:variation-to-besov}
and its sharpness assertions. Section~\ref{sec:conclusion} concludes the
paper and discusses several directions for further research.

Throughout, $A\lesssim B$ means $A\leq CB$ for a constant independent of
the functions under consideration, and $A\simeq B$ means both
$A\lesssim B$ and $B\lesssim A$.

\section{Preliminaries}\label{sect:prelim}
We first recall the atomic definition and measure representation of the
variation space. We then specify the Littlewood--Paley resolution used to
describe Besov regularity. These preliminaries also identify the critical
smoothness index that will govern the forward embedding.

Throughout this paper, $\Omega\subset\mathbb{R}^d$ denotes a bounded Lipschitz domain and $k\in\N_+$. All function spaces are taken over the
real field. We write
\[
\sigmaK(t):=\max\{t,0\}^k
\]
and choose $c>0$ such that
\[
-c<\inf_{\substack{x\in\Omega\\ \omega\in\Sph^{d-1}}}\omega\cdot x
<\sup_{\substack{x\in\Omega\\ \omega\in\Sph^{d-1}}}\omega\cdot x<c.
\]
The normalized $\operatorname{ReLU}^k$ dictionary is
\begin{equation*}
\mathbb{D}:=\left\{\sigmaK(\omega\cdot x-b):\omega\in\Sph^{d-1},\ b\in[-c,c]\right\}.
\end{equation*}
The normalization $|\omega|=1$ removes the rescaling ambiguity in the
outer $\ell^1$ coefficient budget. 

\begin{definition}
Let $X$ be a Banach space of some class of functions on $\Omega$, continuously embedded into the space $\mathcal{D}'(\Omega)$ of distributions.  Let $\mathbb{D}\subset X$ be compact. The
closed absolutely convex hull of $\mathbb{D}$ is
\[
\mathcal K(\mathbb{D})=\mathcal K_X(\mathbb{D}):=
\overline{\left\{\sum_{i=1}^N a_i v_i:
v_i\in\mathbb{D},\ \sum_{i=1}^N|a_i|\leq1\right\}}^{\,X}.
\]
The variation space is defined as
\begin{align*}
\mathcal{L}_1(\mathbb{D})=\mathcal{L}_1^X(\mathbb{D})
&:=\{f\in X:\|f\|_{\mathcal{L}_1(\mathbb{D})}<\infty\}
\end{align*}
with the corresponding norm
\begin{align*}
    \|f\|_{\mathcal{L}_1(\mathbb{D})}
&:=\inf\{t>0:f\in t\mathcal K(\mathbb{D})\}.
\end{align*}
\end{definition}

In the terminology of greedy approximation, $\mathcal K(\mathbb{D})$ is the
atomic convex-hull class usually denoted by $A_1(\mathbb{D})$. Thus
$\mathcal L_1(\mathbb{D})$ is exactly the atomic space whose norm controls the
Hilbert-space greedy approximation rate
\cite{Jones1992,DeVoreTemlyakov1996,BarronCohenDahmenDeVore2008}. We use the standard measure representation of the variation norm from
\cite[Section~2]{SiegelXuVariation}.

\begin{lemma}\label{lem:VariationNormEquivalence}
Let $X$ be a Banach space and suppose that $\mathbb{D}\subset X$ is
compact. Then
$f\in\mathcal L_1(\mathbb{D})$ if and only if there exists a finite
signed Borel measure $\mu$ on $\mathbb{D}$ such that
\[
f=\int_{\mathbb{D}}v\,\dd\mu(v),
\]
and in this case, the variation norm is given by
\begin{equation*}
\|f\|_{\mathcal L_1(\mathbb{D})}
=\inf\left\{\|\mu\|_{\mathrm{TV}}:
f=\int_{\mathbb{D}}v\,\dd\mu(v)\right\}.
\end{equation*}
\end{lemma}
In Lemma \ref{lem:VariationNormEquivalence}, $f$ is viewed as a distribution in $\mathcal{D}^\prime(\Omega)$, and $f=\int_{\mathbb{D}}v\,\dd\mu(v)$ is a Bochner integral, meaning $\langle f,\varphi\rangle=\int_{\mathbb{D}}\langle v,\varphi\rangle\,\dd\mu(v)$ for all $\varphi\in C_c^\infty(\Omega)$.

\subsection{Choice of the ambient space}\label{subsec:ambient-space}
The ambient $X$-norm specifies the topology used to close the absolutely convex hull, while the variation norm is the infimum of the total
variation of representing signed measures. The following result shows that $\mathcal{L}_1^X(\mathbb{D})$ is largely independent of the ambient space $X$.
\begin{proposition}\label{prop:L1D_invariance}
Let $X$ and $Y$ be Banach spaces of functions on $\Omega$, both continuously embedded into $\mathcal{D}'(\Omega)$. Suppose $\mathbb{D}\subset X\cap Y$ is compact in  $X$ and $Y$. Then 
\begin{equation*}
\mathcal L_1^X(\mathbb{D})=\mathcal L_1^Y(\mathbb{D})\quad\text{with identical
variation norms}.
\end{equation*}
Consequently, $\mathcal K_X(\mathbb{D})=\mathcal K_Y(\mathbb{D})$, and
\begin{equation*}
 \|f\|_X\leq \sup_{v\in\mathbb{D}}\|v\|_X\|f\|_{\mathcal L_1(\mathbb{D})}.
\end{equation*}
\end{proposition}
\begin{proof}
By compactness and the continuous embeddings into $\mathcal D'(\Omega)$, the $X$- and $Y$-topologies coincide on $\mathbb{D}$. Thus both realizations
of $\mathbb{D}$ have the same finite signed Borel measures. Let $\mu$ be any such measure and define
\[
 T_X\mu:=\int_{\mathbb{D}}v\,\dd\mu(v)\quad\text{in }X,
 \qquad
 T_Y\mu:=\int_{\mathbb{D}}v\,\dd\mu(v)\quad\text{in }Y.
\]
For every $\varphi\in C_c^\infty(\Omega)$, the distributional pairing
is a continuous linear functional on both spaces, so
\[
 \langle T_X\mu,\varphi\rangle
 =\int_{\mathbb{D}}\langle v,\varphi\rangle\,\dd\mu(v)
 =\langle T_Y\mu,\varphi\rangle.
\]
Consequently, $T_X\mu=T_Y\mu$ in $\mathcal D'(\Omega)$. Lemma~\ref{lem:VariationNormEquivalence} yields
$\mathcal L_1^X(\mathbb{D})=\mathcal L_1^Y(\mathbb{D})$ and, for every
function in this common space,
\[
\begin{aligned}
 \|f\|_{\mathcal L_1^X(\mathbb{D})}
 &=\inf_{\mu:\,T_X\mu=f}\|\mu\|_{\mathrm{TV}}\\
 &=\inf_{\mu:\,T_Y\mu=f}\|\mu\|_{\mathrm{TV}}
 =\|f\|_{\mathcal L_1^Y(\mathbb{D})}.
\end{aligned}
\]

For either ambient space $Z\in\{X,Y\}$,
\[
 \mathcal K_Z(\mathbb{D})
 =\{f\in\mathcal L_1^Z(\mathbb{D}):\|f\|_{\mathcal L_1^Z(\mathbb{D})}\leq1\}.
\]
Equality of the variation spaces
and norms therefore implies
\[
 \mathcal K_X(\mathbb{D})=\mathcal K_Y(\mathbb{D}).
\]
Finally, using the basic inequality
\[
 \left\|\int_{\mathbb{D}}v\,\dd\mu(v)\right\|_X
 \leq \sup_{v\in\mathbb{D}}\|v\|_X\|\mu\|_{\mathrm{TV}},
\]
and then taking the infimum over representing measures proves the embedding.
\end{proof}

We use $L^2(\Omega)$ as the
reference ambient space and fix
$\mathcal L_1(\mathbb{D})=\mathcal L_1^{L^2(\Omega)}(\mathbb{D})$.
We write $W^{s,p}(\Omega)$ for the Sobolev space of order $s\geq0$,
with $W^{0,p}(\Omega)=L^p(\Omega)$. The indices $s,p$ in the following
corollary describe the ambient space and are independent of the Besov
indices in the main embedding theorems.
\begin{corollary}\label{cor:ambient-sobolev}
For $1\leq p\leq\infty$ and $0\leq s<k+1/p$,
let $X=W^{s,p}(\Omega)$.
Then $\mathbb{D}$ is compact in $X$ and, for all such $X$,
\begin{equation*}
\mathcal L_1^X(\mathbb{D})=\mathcal L_1(\mathbb{D})\quad\text{with identical
variation norms}.
\end{equation*}
Consequently, $\mathcal K_X(\mathbb{D})=\mathcal K_{L^2(\Omega)}(\mathbb{D})$, and
\begin{equation*}
 \|f\|_X\leq \sup_{v\in\mathbb{D}}\|v\|_X\|f\|_{\mathcal L_1(\mathbb{D})}.
\end{equation*}
\end{corollary}

\begin{proof}
Set $\Theta:=\Sph^{d-1}\times[-c,c]$ and
$\Phi_{\omega,b}(x):=\sigma_k(\omega\cdot x-b)$. Since $\Omega$ is bounded and $k\geq1$,
the parameter map $(\omega,b)\mapsto\Phi_{\omega,b}$
is continuous into $L^p(\Omega)$ for every $1\leq p\leq\infty$, proving compactness of $\mathbb{D}$ in these spaces.

Choose a cut-off function $\chi\in C_c^\infty(\mathbb{R}^d)$ equal to one near
$\overline\Omega$ and set $F_{\omega,b}:=\chi\Phi_{\omega,b}$.
For $1\leq p<\infty$, the weak derivatives of $F_{\omega,b}$ through
order $k$ are uniformly bounded, have a common compact support, and
converge pointwise away from the limiting hyperplane as the parameters
converge. Dominated convergence gives continuity into
$W^{k,p}(\mathbb{R}^d)$.
For $|\alpha|=k$, the only discontinuous term in $D^\alpha F_{\omega,b}$
is $k!\chi\omega^\alpha\mathbf1_{\{\omega\cdot x>b\}}$.
Its translates differ across a strip of volume $O(|h|)$ on the fixed
support, while the remaining terms are uniformly Lipschitz. Thus,
writing $\tau_hv(x):=v(x+h)$, we have uniformly in $(\omega,b)$
\[
 \|\tau_hD^\alpha F_{\omega,b}-D^\alpha F_{\omega,b}\|_{L^p(\mathbb{R}^d)}^p
 \lesssim\min\{|h|,1\}.
\]
For convergent parameters $(\omega_n,b_n)\to(\omega,b)$, set
$v_n:=D^\alpha(F_{\omega_n,b_n}-F_{\omega,b})$.
Since $v_n\to0$ in $L^p$, for $0<\theta<1/p$ dominated convergence yields
\[
 [v_n]_{W^{\theta,p}(\mathbb{R}^d)}^p
 =\int_{\mathbb{R}^d}
 \frac{\|\tau_hv_n-v_n\|_{L^p(\mathbb{R}^d)}^p}{|h|^{d+\theta p}}\,\dd h
 \longrightarrow0.
\]
The dominating function is a constant times
$\min\{|h|,1\}|h|^{-d-\theta p}$, which is integrable.
This proves continuity into $W^{k+\theta,p}(\mathbb{R}^d)$; lower orders
follow by Sobolev inclusions on $\mathbb{R}^d$.

For $p=\infty$, the derivatives $D^\alpha F_{\omega,b}$ with
$|\alpha|\leq k-1$ are uniformly Lipschitz and depend continuously on
the parameters in $L^\infty(\mathbb{R}^d)$. For $0<\theta<1$, the estimate
\[
 [g]_{W^{\theta,\infty}(\mathbb{R}^d)}
 \lesssim \|g\|_{L^\infty(\mathbb{R}^d)}^{1-\theta}
 \|\nabla g\|_{L^\infty(\mathbb{R}^d)}^\theta,
\]
applied to differences of these derivatives, gives continuity into
$W^{s,\infty}(\mathbb{R}^d)$ for every $0\leq s<k$.
Restriction to $\Omega$ and compactness of $\Theta$ therefore show that
$\mathbb{D}$ is compact in $W^{s,p}(\Omega)$ in the stated range. 

Combining the proved compactness with Proposition \ref{prop:L1D_invariance} completes the proof.
\end{proof}

The variation space $\mathcal{L}_1(\mathbb{D})$ is a Banach space
\cite[Lemma~1]{SiegelXuVariation}. In
Lemma~\ref{lem:radon-variation}, the $L^2$ case of
Corollary~\ref{cor:ambient-sobolev} controls the norm of the ridge
remainder, allowing us to estimate the polynomial part. In the proof of
Theorem~\ref{Thm:BesovEmbeddingVariation}, completeness turns summability
of the variation norms of the Littlewood--Paley blocks into convergence
in $\mathcal L_1(\mathbb{D})$; the $L^2$ embedding then gives convergence in
distributions and identifies the limit with the original Besov function.

For finite $p$, the strict upper bound $s<k+1/p$ is necessary for the
full dictionary: a ridge crossing the interior has a jump in its $k$th normal derivative, whose fractional seminorm diverges at order $1/p$ for $1<p<\infty$.
For $p=\infty$, the dictionary is bounded but not compact in
$W^{k,\infty}(\Omega)$.
Corollary~\ref{cor:ambient-sobolev} shows that
Theorems~\ref{Thm:BesovEmbeddingVariation}
and~\ref{thm:variation-to-besov} hold with the same Besov thresholds for all
these ambient spaces. Therefore, we can set $X=L^2(\Omega)$ in the proofs.

\subsection{Besov spaces}
Having established the representation and completeness properties of the
target space, we now introduce the source regularity scale.
We use the standard Littlewood--Paley description of Besov spaces from
\cite{Triebel,FrazierJawerth1990}.

We choose a radially nonincreasing function
$\varphi_0\in C_c^\infty(\mathbb{R}^d)$ and define an inhomogeneous
Littlewood--Paley resolution $\{\varphi_j\}_{j=0}^\infty$ by
\begin{align}
&0\le\varphi_0\le1,\qquad
\varphi_0(\xi)=1\quad\text{if }|\xi|\leq1,\qquad
\supp\varphi_0\subset\{\xi:|\xi|\le2\},
\label{eq:LP0}\\
&\begin{aligned}
\varphi_j(\xi)&:=\varphi_0(2^{-j}\xi)-\varphi_0(2^{-j+1}\xi),
\qquad j\geq1,\\
0\leq\varphi_j\leq1,\qquad
\supp\varphi_j&\subset\{\xi:2^{j-1}\le|\xi|\le2^{j+1}\},
\qquad j\ge1,
\end{aligned}
\notag\\
&|D^\alpha\varphi_j(\xi)|\lesssim 2^{-j|\alpha|},
\qquad j\ge0,
\label{eq:der}\\
&\sum_{j=0}^\infty\varphi_j(\xi)=1.
\label{eq:LPsum}
\end{align}
For later use, let
\[
\varphi(\xi):=\varphi_0(\xi)-\varphi_0(2\xi).
\]
By \eqref{eq:LP0}, $\varphi$ vanishes on a neighborhood of the origin.
Consequently, $\varphi^\vee$ has moments of every order equal to zero:
\begin{align*}
 \int_{\mathbb{R}^d}x^\alpha\varphi^\vee(x)\,\dd x=0,
 \qquad \alpha\in\mathbb N_0^d.
\end{align*}
Moreover, $\varphi_j(\xi)=\varphi(2^{-j}\xi)$ for every $j\geq1$, so
Fourier scaling gives
\begin{align*}
\varphi_j^\vee(x)=2^{jd}\varphi^\vee(2^jx).
\end{align*}
For $f\in\mathcal S'(\mathbb{R}^d)$, we define the Littlewood--Paley blocks by
\begin{equation*}
\Delta_j f:=(\varphi_j\widehat f)^\vee=\varphi_j^\vee*f.
\end{equation*}
Equation~\eqref{eq:LPsum} gives
\begin{equation}\label{eq:BesovSum}
f=\sum_{j=0}^{\infty}\Delta_j f\qquad\text{in }\mathcal S'(\mathbb{R}^d).
\end{equation}
For $s>0$ and $0<p,q\leq\infty$, the Besov quasi-norm is
\begin{equation*}
\|f\|_{B^s_{p,q}(\mathbb{R}^d)}
:=
\begin{cases}
\displaystyle
\left(\sum_{j\ge0}2^{jsq}\|\Delta_jf\|_{L^p(\mathbb{R}^d)}^q\right)^{1/q},
&0<q<\infty,\\[2mm]
\displaystyle
\sup_{j\ge0}2^{js}\|\Delta_jf\|_{L^p(\mathbb{R}^d)},
&q=\infty.
\end{cases}
\end{equation*}

For a bounded Lipschitz domain $\Omega$, we define the restriction space by
\begin{equation*}
B^s_{p,q}(\Omega)
:=\{f\in\mathcal D'(\Omega):
f=f_e|_\Omega\quad\text{for some }f_e\in B^s_{p,q}(\mathbb{R}^d)\},
\end{equation*}
with quotient quasi-norm
\begin{equation}
\label{eq:Besov-domain-norm}
\|f\|_{B^s_{p,q}(\Omega)}
:=\inf_{\substack{f_e\in B^s_{p,q}(\mathbb{R}^d)\\f_e|_\Omega=f}}
\|f_e\|_{B^s_{p,q}(\mathbb{R}^d)}.
\end{equation}
For brevity, we set $s_0:=k+d/p$.

The definition reduces the forward embedding to controlling one dyadic
block at a time and then summing the resulting variation norms. The next
section develops precisely the local representation needed for the
single-block estimate.

\section{Localization of one Littlewood--Paley block}

This section develops the main ingredients for the proof of
Theorem~\ref{Thm:BesovEmbeddingVariation}. The argument has two parts. We
first decompose a band-limited function into spatial atoms with an
$\ell^p$ coefficient bound. We then estimate the variation norm of each
atom through a Fourier--Radon representation. We begin with the classical
Peetre maximal estimate for band-limited functions
\cite{Peetre1975,Peetre1976}.

Before stating the estimate, we fix the pointwise convention used below.
For $g\in\mathcal S'(\mathbb{R}^d)$, write
$\widehat g:=\mathcal Fg$ for its Fourier transform. If $\widehat g$ has
compact support, inverse Fourier transformation identifies $g$ with a
canonical smooth, slowly growing representative. Whenever $g$ also has
an $L^p$ representative, the two agree almost everywhere. All pointwise
derivatives below refer to this smooth representative.

\begin{lemma}
\label{lem:Peetre}
Let $0<p<\infty$, $L\in\mathbb N_0$, and $N>d/p$. Suppose that
$j\in\mathbb N_0$ and $g\in\mathcal S'(\mathbb{R}^d)$ has an
$L^p$ representative satisfying, for some $C>0$,
\begin{equation*}
\supp\widehat g\subset B(0,C2^j).
\end{equation*}
We define the Peetre maximal function by
\begin{equation*}
\mathcal{P}_{j,N}^Lg(x)
:=
\sup_{y\in\mathbb{R}^d}
\displaystyle\sum_{|\alpha|\le L}\frac{2^{-j|\alpha|}|D^\alpha g(y)|}
{(1+2^j|x-y|)^N}.
\end{equation*}
The following estimate holds:
\begin{equation*}
\|\Pmax_{j,N}^Lg\|_{L^p(\mathbb{R}^d)}\lesssim\|g\|_{L^p(\mathbb{R}^d)}.
\end{equation*}
\end{lemma}

\begin{proof}
We choose $0<r<p$ such that $Nr>d$. The standard derivative form of
Peetre's estimate for band-limited functions gives
\begin{equation*}
\mathcal{P}_{j,N}^Lg(x)
\le C\bigl[\MHL(|g|^r)(x)\bigr]^{1/r},
\end{equation*}
where $\MHL$ is the Hardy--Littlewood maximal operator. See
\cite{Stein1970,Triebel}. Since $p/r>1$, it follows from the
Hardy--Littlewood maximal theorem that
\[
\|\Pmax_{j,N}^Lg\|_{L^p(\mathbb{R}^d)}
\lesssim\|\MHL(|g|^r)\|_{L^{p/r}(\mathbb{R}^d)}^{1/r}
\lesssim\||g|^r\|_{L^{p/r}(\mathbb{R}^d)}^{1/r}
=\|g\|_{L^p(\mathbb{R}^d)},
\]
which proves the estimate of the lemma.
\end{proof}

The next lemma converts a band-limited function into spatially localized
atoms at the reciprocal frequency scale. This is the form of atomic
localization needed below. Compare the classical $\varphi$-transform and
atomic decompositions in \cite{FrazierJawerth1990,Triebel}.

\begin{lemma}
\label{lem:LocalizedAtomDecomposition}
Let $0<p<\infty$ and $j,L\in\mathbb N_0$. Suppose that
$g\in\mathcal S'(\mathbb{R}^d)$ has an $L^p$ representative and satisfies
\begin{equation*}
\supp\widehat g\subset B(0,c2^j)
\end{equation*}
for some $c>0$. Then $g$ admits a decomposition
\begin{equation*}
g=\sum_{\nu\in\mathbb Z^d}\lambda_{j,\nu}a_{j,\nu},
\end{equation*}
with a constant $C>0$ independent of $j$ and $\nu$ such that
\begin{align}
\supp a_{j,\nu}&\subset B(2^{-j}\nu,C2^{-j}),
\label{eq:a-support}\\
|D^\alpha a_{j,\nu}(x)|&\lesssim2^{j|\alpha|},
\qquad |\alpha|\leq L,
\label{eq:a-derivative}\\
\left(\sum_{\nu\in\mathbb{Z}^d}|\lambda_{j,\nu}|^p\right)^{1/p}
&\lesssim 2^{jd/p}\|g\|_{L^p}.
\label{eq:lambda-lp}
\end{align}
\end{lemma}

\begin{proof}
We choose $\eta\in C_c^\infty(\mathbb{R}^d)$ such that
\begin{equation}
\label{eq:eta-POU}
\sum_{\nu\in\mathbb Z^d}\eta(x-\nu)=1,
\qquad x\in\mathbb{R}^d.
\end{equation}
For each $\nu\in\mathbb Z^d$, we set
\[
\eta_{j,\nu}(x):=\eta(2^j(x-\nu_j)),
\qquad \nu_j:=2^{-j}\nu.
\]
We fix $N>d/p$ and define
\begin{align*}
\lambda_{j,\nu}:=\Pmax_{j,N}^{L}g(\nu_j).
\end{align*}
If $\lambda_{j,\nu}=0$ for some $\nu$, its defining supremum forces
$g\equiv0$, and the conclusion is immediate. We may therefore assume
that every $\lambda_{j,\nu}>0$ and set
\begin{align*}
a_{j,\nu}(x):=\frac{\eta_{j,\nu}(x)g(x)}{\lambda_{j,\nu}}.
\end{align*}
The partition of unity \eqref{eq:eta-POU} gives the asserted
decomposition, and the compact support of $\eta$ gives
\eqref{eq:a-support}.

If $x\in\supp\eta_{j,\nu}$, then $2^j|x-\nu_j|\leq C$, and the definition
of the Peetre maximal function therefore implies
\begin{equation}
\label{eq:g-der-local}
|D^\gamma g(x)|\lesssim2^{j|\gamma|}\lambda_{j,\nu},
\qquad |\gamma|\le L.
\end{equation}
The scaled cutoff also satisfies
\begin{align}\label{eq:lambda_nu_j}
|D^\beta\eta_{j,\nu}(x)|\lesssim2^{j|\beta|}.
\end{align}

Combining Leibniz's rule with \eqref{eq:g-der-local} and
\eqref{eq:lambda_nu_j}, we obtain
\begin{align*}
|D^\alpha a_{j,\nu}(x)|
&\le \lambda_{j,\nu}^{-1}
\sum_{\beta\le\alpha}\binom{\alpha}{\beta}
|D^\beta\eta_{j,\nu}(x)|
|D^{\alpha-\beta}g(x)|\lesssim 2^{j|\alpha|},
\end{align*}
which proves \eqref{eq:a-derivative}.

It remains to prove \eqref{eq:lambda-lp}. Set
\[
\mathbb{I}_{j,\nu}:=\nu_j+2^{-j}[0,1)^d.
\]
If $x\in\mathbb I_{j,\nu}$, the definition of the Peetre maximal
function and the triangle inequality
\begin{align*}
    |x-y|\leq |x-\nu_j|+|\nu_j-y|
\end{align*}
imply the bound
\[
\lambda_{j,\nu}\le C\Pmax_{j,N}^{L}g(x),
\]
and integration over $\mathbb{I}_{j,\nu}$ yields
\[
2^{-jd}|\lambda_{j,\nu}|^p
\le C\int_{\mathbb{I}_{j,\nu}}|\Pmax_{j,N}^{L}g(x)|^p\,\dd x.
\]
Summing over $\nu\in\mathbb Z^d$ and applying
Lemma~\ref{lem:Peetre}, we obtain
\[
\sum_{\nu\in\mathbb{Z}^d}|\lambda_{j,\nu}|^p
\lesssim2^{jd}\|\Pmax_{j,N}^{L}g\|_{L^p(\mathbb{R}^d)}^p
\lesssim2^{jd}\|g\|_{L^p(\mathbb{R}^d)}^p,
\]
and taking the $p$-th root completes the proof.
\end{proof}

The preceding decomposition isolates the spatial coefficients that produce
the critical factor $2^{jd/p}$. To control the variation norm of the
resulting atoms, we next connect smooth compactly supported functions with
ridge representations. The following estimate is the required
Fourier--Radon formula. See
\cite{MaoSiegelXu2024,Unser2023} for related ridge and Radon-transform
representations.

\begin{lemma}\label{lem:radon-variation}
Let $\Omega\subset\mathbb{R}^d$ be a bounded Lipschitz domain and let
$f\in C_c^\infty(\mathbb{R}^d)$. Define the filtered Radon profile by
\begin{equation*}
F_{\omega}(t)
:=
\int_{\mathbb{R}}|r|^{d-1}e^{\sqrt{-1}rt}\widehat{f}(r\omega)\,\dd r
\qquad
(\omega,t)\in\mathbb{S}^{d-1}\times\mathbb{R}.
\end{equation*}
Its $(k+1)$-th derivative is
\begin{equation*}
F_\omega^{(k+1)}(t)
=\int_{\mathbb{R}}(\sqrt{-1}r)^{k+1}|r|^{d-1}e^{\sqrt{-1}rt}
\widehat f(r\omega)\,\dd r.
\end{equation*}
The following estimate holds:
\begin{equation*}
\|f\|_{\mathcal L_1(\mathbb{D})}
\lesssim
\|F_\omega^{(k+1)}\|_{L^1(\Sph^{d-1}\times\mathbb{R})}
+\|f\|_{L^2(\Omega)}.
\end{equation*}
\end{lemma}
\begin{proof}
Fourier inversion and polar coordinates give
\begin{align*}
f(x)
&=\frac{1}{(2\pi)^d}\int_{\mathbb{R}^d}
 e^{\sqrt{-1}x\cdot\xi}\widehat f(\xi)\,\dd\xi\\
&=\frac{1}{2(2\pi)^d}\int_{\Sph^{d-1}}\int_{\mathbb{R}}
 |r|^{d-1}e^{\sqrt{-1}r\omega\cdot x}\widehat f(r\omega)
\,\dd r\,\dd\omega\\
&=\frac{1}{2(2\pi)^d}\int_{\Sph^{d-1}}
F_\omega(\omega\cdot x)\,\dd\omega.
\end{align*}
We choose $R\in(0,c)$ so that
\[
 |\omega\cdot x|<R,
 \qquad x\in\Omega,\quad \omega\in\Sph^{d-1},
\]
which is possible by the choice of $c$. We set
$C_d:=1/[2(2\pi)^d]$. Taylor's formula with integral remainder, applied
to $F_\omega$ at $-R$ (see \cite{DeVoreLorentz1993}), gives, for
$x\in\Omega$,
\begin{align*}
f(x)={}&C_d\int_{\Sph^{d-1}}\sum_{j=0}^k
\frac{F_\omega^{(j)}(-R)}{j!}(\omega\cdot x+R)^j\,\dd\omega\\
&+\frac{C_d}{k!}\int_{\Sph^{d-1}}\int_{-R}^{\omega\cdot x}
F_\omega^{(k+1)}(t)(\omega\cdot x-t)^k\,\dd t\,\dd\omega.
\end{align*}
We define the polynomial part by
\[
p(x):=C_d\int_{\Sph^{d-1}}\sum_{j=0}^k
\frac{F_\omega^{(j)}(-R)}{j!}(\omega\cdot x+R)^j\,\dd\omega.
\]
It follows from this definition that $p$ is a polynomial of degree at
most $k$. Write the remainder as
\[
g(x)=\frac{C_d}{k!}\int_{\Sph^{d-1}}\int_{-R}^{R}
F_\omega^{(k+1)}(t)\sigma_k(\omega\cdot x-t)
\,\dd t\,\dd\omega.
\]
Because $f$ is real-valued, conjugate symmetry of $\widehat f$ shows that
$F_\omega^{(k+1)}$ is real-valued. Push the signed measure with density
$(C_d/k!)F_\omega^{(k+1)}(t)$ on
$\Sph^{d-1}\times[-R,R]$ forward under
$(\omega,t)\mapsto\sigma_k(\omega\cdot{}-t)$. Its total variation does not
increase, so Lemma~\ref{lem:VariationNormEquivalence} gives
\[
\|g\|_{\mathcal L_1(\mathbb{D})}
\lesssim\int_{\Sph^{d-1}}\int_{-R}^{R}
|F_\omega^{(k+1)}(t)|\,\dd t\,\dd\omega.
\]
The restrictions to $\Omega$ of polynomials of degree at most $k$ form
a subspace of $\mathcal L_1(\mathbb{D})$. Indeed, whenever
$b<\inf_{x\in\Omega}\omega\cdot x$, the dictionary element
$\sigma_k(\omega\cdot x-b)$ agrees on $\Omega$ with the polynomial
$(\omega\cdot x-b)^k$. We choose distinct values
\[
 b_0,\ldots,b_k\in
 \left(-c,\inf_{\substack{x\in\Omega\\
 \omega\in\Sph^{d-1}}}\omega\cdot x\right).
\]
For a suitable finite set of directions $\omega$, the Vandermonde
identity in $b$, together with polarization in $\omega$,
shows that these ridge polynomials span every polynomial of degree at
most $k$. The equivalence of norms on this finite-dimensional space
therefore gives
\[
 \|p\|_{\mathcal L_1(\mathbb{D})}\lesssim\|p\|_{L^2(\Omega)}.
\]
Since $p=f-g$ on $\Omega$ and
$\|g\|_{L^2(\Omega)}\lesssim\|g\|_{\mathcal L_1(\mathbb{D})}$, the preceding
estimates imply
\[
\begin{aligned}
 \|f\|_{\mathcal L_1(\mathbb{D})}
 &\leq\|p\|_{\mathcal L_1(\mathbb{D})}
 +\|g\|_{\mathcal L_1(\mathbb{D})}\\
 &\lesssim \|f\|_{L^2(\Omega)}
 +\int_{\Sph^{d-1}}\int_{\mathbb{R}}
 |F_\omega^{(k+1)}(t)|\,\dd t\,\dd\omega,
\end{aligned}
\]
which completes the proof.
\end{proof}

We next derive the variation estimate for a smooth function localized at
scale $h>0$ by first proving a one-dimensional inverse Fourier estimate
and then applying the Fourier--Radon representation from
Lemma~\ref{lem:radon-variation}.

\begin{lemma}
\label{lem:1D-Fourier-L1}
For every $\Phi\in H^1(\mathbb{R})$, we have
\begin{equation*}
\|\Phi^\vee\|_{L^1(\mathbb{R})}
\lesssim\bigl(\|\Phi\|_{L^2(\mathbb{R})}+\|\Phi'\|_{L^2(\mathbb{R})}\bigr).
\end{equation*}
\end{lemma}

\begin{proof}
By the Cauchy--Schwarz inequality and Plancherel's identity,
\begin{align*}
\|\Phi^\vee\|_{L^1(\mathbb{R})}
&=\int_{\mathbb{R}}(1+t^2)^{-1/2}(1+t^2)^{1/2}|\Phi^\vee(t)|\,\dd t\\
&\lesssim\left(\int_{\mathbb{R}}(1+t^2)|\Phi^\vee(t)|^2\,\dd t\right)^{1/2}\\
&\simeq(\|\Phi\|_{L^2(\mathbb{R})}^2+\|\Phi^\prime\|_{L^2(\mathbb{R})}^2)^{\frac{1}{2}},
\end{align*}
which proves the inequality of the lemma.
\end{proof}

The one-dimensional estimate now converts the scaled Fourier profiles into
an $L^1$ bound. Combined with the Fourier--Radon formula, it yields the
uniform variation estimate for a localized atom.

\begin{lemma}
\label{lem:atom-variation}
Fix $A>0$ and an integer $R_0>k+d+2$. Let $0<h\leq1$,
$x_0\in\mathbb{R}^d$, and $\psi\in C_c^\infty(\mathbb{R}^d)$ satisfying
$\supp\psi\subset B(0,A)$. Set
\begin{equation*}
\psi_h(x):=\psi\left(\frac{x-x_0}{h}\right).
\end{equation*}
The following estimate holds:
\begin{equation}\label{eq:uniform-atom-variation}
\|\psi_h\|_{\mathcal{L}_1(\mathbb{D})}
\lesssim h^{-k}
\max_{|\alpha|\leq R_0}\|D^\alpha\psi\|_{L^\infty(\mathbb{R}^d)}.
\end{equation}
The implicit constant may depend on $d$, $k$, $A$, $R_0$, and $\Omega$,
but is independent of $h$, $x_0$, and $\psi$.
\end{lemma}

\begin{proof}
We write
\[
 M_{R_0}(\psi):=
 \max_{|\alpha|\leq R_0}\|D^\alpha\psi\|_{L^\infty(\mathbb{R}^d)}.
\]
We define the one-dimensional Radon profile by
\begin{align*}
 F_{\omega,h}(t):=\int_{\mathbb{R}}|r|^{d-1}
 \widehat{\psi}_h(r\omega)e^{\sqrt{-1}rt}\,\dd r.
\end{align*}
Differentiating $F_{\omega,h}$ $k+1$ times in $t$ gives
\begin{equation}
\label{eq:F-derivative}
F_{\omega,h}^{(k+1)}(t)
=\int_{\mathbb{R}}(\sqrt{-1}r)^{k+1}|r|^{d-1}\widehat \psi_h(r\omega)e^{\sqrt{-1}rt}\,\dd r.
\end{equation}
A change of variables gives
\[
 \|\psi_h\|_{L^2(\Omega)}
 \leq h^{d/2}\|\psi\|_{L^2(\mathbb{R}^d)}
 \lesssim h^{-k}M_{R_0}(\psi),
\]
so Lemma~\ref{lem:radon-variation} reduces the proof to showing that
\[
\int_{\Sph^{d-1}}
\|F_{\omega,h}^{(k+1)}\|_{L^1(\mathbb{R})}\,\dd\omega
\lesssim h^{-k}M_{R_0}(\psi).
\]

\textbf{Step 1: Exact profile scaling.}
The Fourier scaling of the definition of $\psi_h$ is
\begin{equation}
\label{eq:ahat-scaling}
\widehat \psi_h(\xi)=h^d e^{-\sqrt{-1}x_0\cdot\xi}\widehat \psi(h\xi).
\end{equation}
For $\omega\in\Sph^{d-1}$, define
\begin{align*}
M_\omega(s):=(\sqrt{-1}s)^{k+1}|s|^{d-1}\widehat \psi(s\omega),
\end{align*}
so that substituting \eqref{eq:ahat-scaling} into \eqref{eq:F-derivative} and
changing variables $s=hr$ gives
\begin{equation*}
F_{\omega,h}^{(k+1)}(t)
=2\pi h^{-k-1}
M_\omega^\vee\left(\frac{t-\omega\cdot x_0}{h}\right),
\end{equation*}
and consequently
\begin{equation}
\label{eq:F-L1-scaling}
\|F_{\omega,h}^{(k+1)}\|_{L^1(\mathbb{R})}
=2\pi h^{-k}\|M_\omega^{\vee}\|_{L^1(\mathbb{R})}.
\end{equation}
It remains only to prove a unit-scale estimate uniformly in
$\omega\in\Sph^{d-1}$.

\textbf{Step 2: Uniform unit-scale estimate.}
Since $\supp\psi\subset B(0,A)$, integration by parts up to order
$R_0$ gives
\begin{align*}
\begin{aligned}
|\widehat\psi(s\omega)|&\lesssim
M_{R_0}(\psi)(1+|s|)^{-R_0},\\
|\nabla_\xi\widehat\psi(s\omega)|&\lesssim
M_{R_0}(\psi)(1+|s|)^{-R_0},
\end{aligned}
\end{align*}
which implies uniformly in $\omega$ that
\begin{align*}
|M_{\omega}(s)|
&\lesssim M_{R_0}(\psi)|s|^{k+d}(1+|s|)^{-R_0},\\
|M'_\omega(s)|
&\lesssim M_{R_0}(\psi)
\bigl(|s|^{k+d-1}+|s|^{k+d}\bigr)(1+|s|)^{-R_0}.
\end{align*}
The choice $R_0>k+d+2$ therefore gives
\begin{align*}
\sup_{\omega\in\Sph^{d-1}}
\left(\|M_\omega\|_{L^2(\mathbb{R})}+\|M'_\omega\|_{L^2(\mathbb{R})}\right)
\lesssim M_{R_0}(\psi).
\end{align*}
Combining this estimate with Lemma~\ref{lem:1D-Fourier-L1} yields
\begin{equation*}
\sup_{\omega\in\Sph^{d-1}}\|M_\omega^{\vee}\|_{L^1(\mathbb{R})}
\lesssim M_{R_0}(\psi).
\end{equation*}
Since $\Sph^{d-1}$ has finite measure,
\eqref{eq:F-L1-scaling} gives
\begin{equation*}
\int_{\Sph^{d-1}}\int_{\mathbb{R}}|F_{\omega,h}^{(k+1)}(t)|\,\dd t\,\dd\omega
\lesssim h^{-k}M_{R_0}(\psi).
\end{equation*}
This completes the proof.
\end{proof}

We have therefore obtained both components needed in the forward proof:
Lemma~\ref{lem:LocalizedAtomDecomposition} supplies spatial atoms and their
coefficient bound, while Lemma~\ref{lem:atom-variation} controls the
variation norm of every atom after rescaling. We now assemble these
estimates across locations and dyadic scales.

\section{Forward embedding and sharpness}

The proof of Theorem \ref{Thm:BesovEmbeddingVariation} proceeds in two
steps. We apply the atomic decomposition and the single-atom variation
estimate directly to a near-minimal Besov extension at each dyadic scale,
and then sum over the scales in the critical and above-critical regimes.

\begin{proof}
We fix $f\in B_{p,q}^s(\Omega)$. By \eqref{eq:Besov-domain-norm}, for every
$\varepsilon>0$ there exists $f_e\in B_{p,q}^s(\mathbb{R}^d)$ such that
$f_e|_\Omega=f$ and
\[
\|f_e\|_{B_{p,q}^s(\mathbb{R}^d)}
\leq \|f\|_{B_{p,q}^s(\Omega)}+\varepsilon.
\]
Since the variation norm depends only on the restriction to $\Omega$,
\begin{equation*}
\|f\|_{\mathcal L_1(\mathbb{D})}
=\|f_e|_\Omega\|_{\mathcal L_1(\mathbb{D})}.
\end{equation*}
We therefore estimate the latter norm directly.

\textbf{Step 1: Dyadic atomic estimate.}
We apply Lemma~\ref{lem:LocalizedAtomDecomposition} to each
$f_j:=\Delta_jf_e$, using the integer $L=R_0>k+d+2$ from
Lemma~\ref{lem:atom-variation}. Retaining only the atoms whose supports
meet $\Omega$, we obtain on $\Omega$
\begin{align*}
f_j|_\Omega
=\sum_{\substack{\nu\in\mathbb Z^d\\
\supp a_{j,\nu}\cap\Omega\neq\varnothing}}
\lambda_{j,\nu}a_{j,\nu}|_\Omega.
\end{align*}
All sums over $\nu$ below are restricted to the retained atoms. As
$\Omega$ is bounded, this index set is finite for every fixed $j$.
For each atom, we have
\begin{equation*}
\|a_{j,\nu}|_\Omega\|_{\mathcal L_1(\mathbb{D})}\lesssim2^{jk}.
\end{equation*}
Indeed, by \eqref{eq:a-support} and \eqref{eq:a-derivative}, the rescaled
function
\[
y\longmapsto a_{j,\nu}\bigl(2^{-j}\nu+2^{-j}y\bigr)
\]
is supported in $B(0,C)$ and satisfies, for every $|\alpha|\leq R_0$,
\[
\left|
D_y^\alpha\left[
a_{j,\nu}\bigl(2^{-j}\nu+2^{-j}y\bigr)
\right]
\right|
=
2^{-j|\alpha|}
\left|
(D^\alpha a_{j,\nu})
\bigl(2^{-j}\nu+2^{-j}y\bigr)
\right|
\lesssim 1.
\]
It follows from the preceding support and derivative estimates that the
rescaled atoms form a bounded subset of
$C_c^{R_0}(B(0,C))$. Moreover, since
$\operatorname{supp}a_{j,\nu}\cap\Omega\neq\varnothing$ and $\Omega$
is bounded, the relevant translation parameters $2^{-j}\nu$ remain
in a fixed bounded set. The quantitative estimate
\eqref{eq:uniform-atom-variation}, with $h=2^{-j}$ and
$x_0=2^{-j}\nu$, therefore yields uniformly in $j$ and $\nu$ that
\[
\|a_{j,\nu}|_\Omega\|_{\mathcal L_1(\mathbb{D})}
\lesssim 2^{jk}.
\]
Using $\ell^p\hookrightarrow\ell^1$ for $0<p\leq1$, we obtain
\begin{equation*}
\|f_j|_\Omega\|_{\mathcal{L}_1(\mathbb{D})}
\lesssim 2^{jk}\sum_{\nu\in\mathbb{Z}^d}|\lambda_{j,\nu}|
\leq 2^{jk}\left(\sum_{\nu\in\mathbb{Z}^d}|\lambda_{j,\nu}|^p\right)^{1/p}
\lesssim 2^{j(k+d/p)}\|f_j\|_{L^p(\mathbb{R}^d)}.
\end{equation*}
Here the last estimate follows from \eqref{eq:lambda-lp}, including at
the fixed scale $j=0$.

\textbf{Step 2: Summation over dyadic scales.}
For $s\geq s_0$ and $0<q\leq1$, the elementary embedding
$\ell^q\hookrightarrow\ell^1$ yields
\begin{equation}
\begin{aligned}\label{eq:critical-sum}
\sum_{j=0}^\infty\|f_j|_\Omega\|_{\mathcal L_1(\mathbb{D})}
&\lesssim\sum_{j=0}^\infty2^{js_0}\|f_j\|_{L^p(\mathbb{R}^d)}\\
&\leq\sum_{j=0}^\infty2^{-j(s-s_0)}
\left(2^{js}\|f_j\|_{L^p(\mathbb{R}^d)}\right)
\leq\|f_e\|_{B_{p,q}^{s}(\mathbb{R}^d)}.
\end{aligned}    
\end{equation}

When $q>1$ and $s>s_0$, H\"older's inequality gives
\begin{equation}\label{eq:gjL1DBound}
    \begin{aligned}
\sum_{j=0}^\infty\|f_j|_\Omega\|_{\mathcal L_1(\mathbb{D})}
&\lesssim\sum_{j=0}^\infty
2^{-j(s-s_0)}2^{js}\|f_j\|_{L^p(\mathbb{R}^d)}\\
&\leq
\left(\sum_{j=0}^\infty2^{jsq}\|f_j\|^q_{L^p(\mathbb{R}^d)}\right)^{\frac{1}{q}}
\left(\sum_{j=0}^\infty2^{-j(s-s_0)q^\prime}\right)^{\frac{1}{q^\prime}}\\
&\lesssim\|f_e\|_{B_{p,q}^s(\mathbb{R}^d)},
\end{aligned}
\end{equation}
where $q'$ is the conjugate exponent of $q$, with the usual modification
when $q=\infty$.

Finally, we set $S_N:=\sum_{j=0}^Nf_j$. By \eqref{eq:critical-sum} when
$0<q\leq1$ and by \eqref{eq:gjL1DBound} when $q>1$, the sequence
$S_N|_\Omega$ is Cauchy in $\mathcal L_1(\mathbb{D})$. By the completeness
\textcolor{blue}{recalled in Subsection~\ref{subsec:ambient-space}}, it converges
to some $h\in\mathcal L_1(\mathbb{D})$. The continuous embedding
$\mathcal L_1(\mathbb{D})\hookrightarrow L^2(\Omega)$ implies that this
convergence also holds in $L^2(\Omega)$,
\[
 S_N|_\Omega\longrightarrow h\qquad\text{in }L^2(\Omega),
\]
thus for every $\zeta\in C_c^\infty(\Omega)$,
\[
\int_\Omega S_N(x)\zeta(x)\,\dd x
\longrightarrow
\int_\Omega h(x)\zeta(x)\,\dd x,
\]
whereas \eqref{eq:BesovSum} gives
$S_N\to f_e$ in $\mathcal S'(\mathbb{R}^d)$, and hence
\[
\int_\Omega S_N(x)\zeta(x)\,\dd x
=\langle S_N,\zeta\rangle
\longrightarrow
\langle f_e,\zeta\rangle.
\]
Comparing these two limits yields
\[
\langle f_e,\zeta\rangle
=\int_\Omega h(x)\zeta(x)\,\dd x
\qquad\text{for every }\zeta\in C_c^\infty(\Omega).
\]
Thus $f_e|_\Omega=h$ in $\mathcal D'(\Omega)$. Since
$f_e|_\Omega=f$, the function $f$ has the $L^2(\Omega)$ representative
$h$. As a consequence,
the Littlewood--Paley expansion converges in the variation space:
\begin{equation*}
f=f_e|_\Omega=\sum_{j=0}^\infty f_j|_\Omega
\qquad\text{in }\mathcal L_1(\mathbb{D}).
\end{equation*}
Combining \eqref{eq:critical-sum} when $q\leq1$ (or
\eqref{eq:gjL1DBound} when $q>1$) with the choice of $f_e$ gives
\begin{equation*}
 \|f\|_{\mathcal L_1(\mathbb{D})}
 \lesssim \|f\|_{B_{p,q}^s(\Omega)}+\varepsilon.
\end{equation*}
We let $\varepsilon\to0$ to complete the proof.
\end{proof}

\subsection{Sharpness of the forward embedding}

The theorem gives the positive embedding at and above the critical index.
It remains to determine whether this threshold can be lowered. We show that
it cannot by comparing the Besov and variation scaling of a compactly
supported bump. To prepare the dyadic estimates, for
$w\in C_c^\infty(\mathbb{R}^d)$ and $L\in\mathbb N_+$, denote the order-$L$
Taylor remainder by
\begin{align}\label{eq:TaylorRemainder}
     R_w^L(x,y)
 :=L\sum_{|\alpha|=L}\frac{(-y)^\alpha}{\alpha!}
 \int_0^1(1-r)^{L-1}
 D^\alpha w(x-r y)\,\dd r.
\end{align}
The next lemma gives the exact two-sided scaling needed for the
counterexample.

\begin{lemma}\label{lem:L1DLowerBound}
Let $0<p,q\leq\infty$ and $s>0$, with the convention $d/\infty=0$.
There exist $x_0\in\Omega$ and
$\psi\in C_c^\infty(\mathbb{R}^d)$ such that, with
\begin{equation*}
\psi_h(x):=\psi\left(\frac{x-x_0}{h}\right).
\end{equation*}
the following estimates hold for all sufficiently small $h>0$:
\begin{equation*}
\|\psi_h\|_{B_{p,q}^s(\mathbb{R}^d)}\simeq h^{d/p-s},\qquad
\|\psi_h\|_{\mathcal L_1(\mathbb{D})}\simeq h^{-k}.
\end{equation*}
\end{lemma}

\begin{proof}
We choose $x_0\in\Omega$, $r_0>0$, and an integer $L$ such that
$B(x_0,2r_0)\subset\Omega$ and
\[
 L>\max\left\{k,s,\frac dp-d\right\}.
\]
We choose a nonzero real-valued $\eta\in C_c^\infty(B(0,r_0))$ and set
$\psi:=\partial_1^L\eta$. Integration by parts shows that
\begin{equation}\label{eq:VanishingCondition}
\int_{\mathbb{R}^d}y^\alpha\psi(y)\,\dd y=0,
\qquad |\alpha|<L.
\end{equation}
For all sufficiently small $h\in(0,1)$, we have
$\supp\psi_h\subset\Omega$.

\textbf{Step 1: Estimate of the variation norm.}

We define the linear functional
\begin{equation}\label{eq:BumpFunctional}
\Lambda_h(v):=h^{-d-k}\int_\Omega
\psi\left(\frac{x-x_0}{h}\right)v(x)\,\dd x.
\end{equation}
For $(\omega,b)\in\Sph^{d-1}\times[-c,c]$, we set
$a:=(\omega\cdot x_0-b)/h$. The change of variables $x=x_0+hy$,
together with the $k$-homogeneity of $\sigma_k$, gives
\begin{equation*}
\Lambda_h\bigl(\sigma_k(\omega\cdot{}-b)\bigr)
=\int_{\mathbb{R}^d}\psi(y)\sigma_k(\omega\cdot y+a)\,\dd y
=:T_\omega(a).
\end{equation*}
If $a\leq-r_0$, then $T_\omega(a)=0$. If $a\geq r_0$, the function
$\sigma_k(a+\omega\cdot y)$ is a polynomial of degree at most $k$ on
$\supp\psi$, and hence \eqref{eq:VanishingCondition} also gives
$T_\omega(a)=0$. For $|a|<r_0$,
\[
|T_\omega(a)|\leq (2r_0)^k\|\psi\|_{L^1(\mathbb{R}^d)},
\]
uniformly over the dictionary. For each fixed $h$, the functional
$\Lambda_h$ is continuous on $L^2(\Omega)$. Its bound on the atoms
therefore extends to their closed absolutely convex hull, and homogeneity
of the gauge gives
\begin{equation}\label{eq:LambdahBound}
|\Lambda_h(v)|\lesssim\|v\|_{\mathcal L_1(\mathbb{D})}.
\end{equation}
This holds for every $v\in\mathcal L_1(\mathbb{D})$. Taking $v=\psi_h$ in
\eqref{eq:BumpFunctional} gives
\[
\Lambda_h(\psi_h)=h^{-k}\int_{\mathbb{R}^d}\psi(y)^2\,\dd y\gtrsim h^{-k}.
\]
Combining this estimate with \eqref{eq:LambdahBound} gives the lower
bound, while Lemma~\ref{lem:atom-variation} gives the matching upper
bound. Thus the second estimate of the lemma is proved.

\textbf{Step 2: Estimate of the Besov norm.}

We prove the first estimate of the lemma by a dyadic argument. We set
$m:=\left\lceil\log_2(1/h)\right\rceil$ and split the estimate into the
ranges $1\leq j<m$ and $j\geq m$. The fixed block
$j=0$ is treated at the end. For $j\geq1$, Fourier scaling gives
\begin{equation}\label{eq:bump-rescaled-block}
 \Delta_j\psi_h(x_0+hx)=(2^jh)^d\int_{\mathbb{R}^d}
 \varphi^\vee\bigl(2^jh(x-y)\bigr)\psi(y)\,\dd y.
\end{equation}
We first suppose that $1\leq j<m$. Taylor's formula for $\varphi^\vee$ gives
\[
 \varphi^\vee\bigl(2^jh(x-y)\bigr)
 =\sum_{|\alpha|<L}
 \frac{(-2^jhy)^\alpha}{\alpha!}
 D^\alpha\varphi^\vee(2^jhx)
 +R_{\varphi^\vee}^L(2^jhx,2^jhy).
\]
Since $\varphi^\vee\in\mathcal S(\mathbb{R}^d)$,
for every $M>0$ \eqref{eq:TaylorRemainder} gives
\[
 \left|R_{\varphi^\vee}^L(2^jhx,2^jhy)\right|
 \lesssim(2^jh)^L|y|^L
 (1+2^jh|x|)^{-M}.
\]
The moment condition \eqref{eq:VanishingCondition} annihilates the Taylor
polynomial in \eqref{eq:bump-rescaled-block}. Hence
\begin{align*}
 |\Delta_j\psi_h(x_0+hx)|
 &=(2^jh)^d\left|\int_{\mathbb{R}^d}
 R_{\varphi^\vee}^L(2^jhx,2^jhy)\psi(y)\,\dd y\right|\\
 &\lesssim(2^jh)^{L+d}(1+2^jh|x|)^{-M}.
\end{align*}
Taking the $L^p$ quasi-norm with $M>d/p$ yields
\begin{equation}\label{eq:bump-low-frequency}
h^{-d/p}\|\Delta_j\psi_h\|_{L^p(\mathbb{R}^d)}
\lesssim(2^jh)^{L+d-d/p}.
\end{equation}
For $j\geq m$, changing variables $z=2^jh(x-y)$ in
\eqref{eq:bump-rescaled-block} gives
\[
 \Delta_j\psi_h(x_0+hx)
 =\int_{\mathbb{R}^d}\varphi^\vee(z)
 \psi\left(x-\frac{z}{2^jh}\right)\,\dd z.
\]
Taylor's formula for $\psi$ at $x$ of order $L$ gives
\begin{align*}
 \psi\left(x-\frac{z}{2^jh}\right)
 =\sum_{|\alpha|<L}\frac{(-z)^\alpha}
 {\alpha!(2^jh)^{|\alpha|}}D^\alpha\psi(x)
 +R_\psi^L\left(x,\frac{z}{2^jh}\right).
\end{align*}
The vanishing moments of $\varphi^\vee$ annihilate the Taylor polynomial,
thus by \eqref{eq:TaylorRemainder} we have
\begin{equation}\label{eq:TaylorRemainderKappa}
    \begin{aligned}
 |\Delta_j\psi_h(x_0+hx)|
 &\leq\int_{\mathbb{R}^d}|\varphi^\vee(z)|
 \left|R_\psi^L\left(x,\frac{z}{2^jh}\right)\right|\,\dd z\\
 &\lesssim(2^jh)^{-L}
 \int_{\mathbb{R}^d}|z|^L|\varphi^\vee(z)|\,\dd z.
\end{aligned}
\end{equation}

We set $\lambda:=2^jh\geq1$. If $|x|\leq2r_0$, then
$(1+|x|)^{-M}\simeq1$. If $|x|>2r_0$ and the integrand in
\eqref{eq:TaylorRemainderKappa} is nonzero, then for some
$r\in(0,1]$ we have
\[
 \left|x-\frac{r z}{2^jh}\right|\leq r_0,
 \qquad
 |z|\geq\lambda(|x|-r_0)\geq\lambda|x|/2.
\]
Consequently, Schwartz decay gives
\begin{align*}
 |\Delta_j\psi_h(x_0+hx)|
 &\lesssim\lambda^{-L}
 \int_{|z|\geq\lambda|x|/2}|z|^L|\varphi^\vee(z)|\,\dd z
 \\
 &\lesssim_M\lambda^{-L}(1+\lambda|x|)^{-M}
 \lesssim_M\lambda^{-L}(1+|x|)^{-M}.
\end{align*}
Combining this with \eqref{eq:TaylorRemainderKappa} for $|x|\leq2r_0$
gives, for every $M>0$,
\begin{align}\label{eq:Deltajkappahbound}
 |\Delta_j\psi_h(x_0+hx)|
 \lesssim(2^jh)^{-L}(1+|x|)^{-M}.
\end{align}
After choosing $M>d/p$, taking the $L^p$ quasi-norm in \eqref{eq:Deltajkappahbound} yields
\begin{equation*}
\|\Delta_j\psi_h\|_{L^p(\mathbb{R}^d)}
\lesssim h^{d/p}(2^jh)^{-L}.
\end{equation*}
For the fixed block, a direct change of variables yields
\[
 \Delta_0\psi_h(x_0+hx)
 =h^d\int_{\mathbb{R}^d}\varphi_0^\vee(h(x-y))\psi(y)\,\dd y.
\]
Taylor expansion of $\varphi_0^\vee(h(x-y))$ in $y$, together with
\eqref{eq:VanishingCondition}, yields
\[
 |\Delta_0\psi_h(x_0+hx)|
 \lesssim_M h^{d+L}(1+h|x|)^{-M},
\]
for every $M>0$. Choosing $M>d/p$ when $p<\infty$ and changing variables
gives $\|\Delta_0\psi_h\|_{L^p(\mathbb{R}^d)}\lesssim h^{d+L}$. For $p=\infty$, the
same estimate follows by taking the supremum. Thus the fixed block also
satisfies \eqref{eq:bump-low-frequency}. Combining the low- and
high-frequency estimates gives
\[
 2^{js}\|\Delta_j\psi_h\|_{L^p(\mathbb{R}^d)}
 \lesssim h^{d/p-s}\cdot
 \begin{cases}
 (2^jh)^{s+L+d-d/p},&j<m,\\
 (2^jh)^{s-L},&j\geq m.
 \end{cases}
\]
For $q<\infty$, the two resulting geometric series converge because
$s+L+d-d/p>0$ and $s-L<0$. For $q=\infty$, the corresponding supremum
is finite. Therefore
\begin{equation}\label{eq:bump-besov-upper}
\|\psi_h\|_{B_{p,q}^s(\mathbb{R}^d)}\lesssim h^{d/p-s}.
\end{equation}

For the lower bound, we set $j_h:=\lceil\log_2(1/h)\rceil$ and
$\tau:=2^{j_h}h\in[1,2]$. By \eqref{eq:bump-rescaled-block},
\[
 \Delta_{j_h}\psi_h(x_0+hx)=T_\tau\psi(x),
 \qquad
 \widehat{T_\tau\psi}(\xi)=\varphi(\xi/\tau)\widehat\psi(\xi).
\]
The function $T_\tau\psi$ is nonzero for every $\tau\in[1,2]$.
Indeed, otherwise the entire function $\widehat\psi$ would vanish on a
nonempty open set on which $\varphi(\cdot/\tau)$ is nonzero, and hence
would vanish identically. Moreover, $\tau\mapsto T_\tau\psi$ is
continuous from $[1,2]$ into $\mathcal S(\mathbb{R}^d)$. Compactness
therefore gives
\[
 \inf_{1\leq\tau\leq2}\|T_\tau\psi\|_{L^p(\mathbb{R}^d)}>0,
\]
also when $p=\infty$. Scaling and the single block $j_h$ now give
\[
\|\psi_h\|_{B_{p,q}^s(\mathbb{R}^d)}
\gtrsim 2^{j_hs}\|\Delta_{j_h}\psi_h\|_{L^p}
\gtrsim h^{d/p-s}.
\]
Combining this estimate with \eqref{eq:bump-besov-upper} proves the first
estimate of the lemma and completes the proof.
\end{proof}

As an immediate consequence, let $0<s<s_0$ and define
\begin{equation*}
f_h:=h^{s-d/p}\psi_h.
\end{equation*}
Then Lemma~\ref{lem:L1DLowerBound} gives
\[
\|f_h\|_{B_{p,q}^s(\Omega)}
\leq\|f_h\|_{B_{p,q}^s(\mathbb{R}^d)}\lesssim1,
\qquad
\|f_h\|_{\mathcal L_1(\mathbb{D})}\gtrsim h^{s-s_0}\longrightarrow\infty
\quad\text{as }h\to0.
\]
Consequently,
$B_{p,q}^s(\Omega)\not\hookrightarrow\mathcal L_1(\mathbb{D})$ whenever
$0<s<s_0$.

\section{Converse embedding and sharpness}

We now reverse the direction of the embedding. The argument first converts
the variation representation into compactly supported extensions whose
derivatives of order $k+1$ are finite measures. A vector-valued
Calder\'on--Zygmund estimate then yields the positive
$B_{p,2}^{k+1}$ inclusion. We subsequently construct two families of ridge
functions: separated ridges rule out larger smoothness, while lacunary
ridges show that the fine index $2$ cannot be decreased.

\subsection{Proof of Theorem~\ref{thm:variation-to-besov}}

We begin with the positive inclusion. The measure representation from
Lemma~\ref{lem:VariationNormEquivalence} provides the derivative bounds,
and the Littlewood--Paley multiplier decomposition converts those bounds
into the required square-function estimate.

\begin{proof}
We set $m:=k+1$ and choose $\chi\in C_c^\infty(\mathbb{R}^d)$ with
$\chi=1$ on a neighborhood of $\overline\Omega$. By
Lemma~\ref{lem:VariationNormEquivalence} and the compact parameterization
of $\mathbb{D}$, for every $\varepsilon>0$ there is a finite signed Borel
measure $\mu$ on $\Sph^{d-1}\times[-c,c]$ such that, almost everywhere
on $\Omega$,
\[
 f(x)=\int_{\Sph^{d-1}\times[-c,c]}
 \sigma_k(\omega\cdot x-b)\,\dd\mu(\omega,b),
 \qquad
 \|\mu\|_{\mathrm{TV}}\leq\|f\|_{\mathcal L_1(\mathbb{D})}+\varepsilon.
\]
To justify this parameterized representation, we approximate $f$ in
$L^2(\Omega)$ by finite atomic sums with uniformly bounded coefficient
sums, viewing their coefficients as measures on the compact parameter
space, and take a weak-* convergent subsequence. Continuity of the
parameter map into $L^2(\Omega)$ identifies the weak-* limit with $f$,
and lower semicontinuity gives the displayed total-variation bound.
We define the cutoff atoms and the extension by
\[
 \chi_{\omega,b}(x):=\chi(x)\sigma_k(\omega\cdot x-b),
 \qquad
 f_e(x):=\int_{\Sph^{d-1}\times[-c,c]}
 \chi_{\omega,b}(x)\,\dd\mu(\omega,b).
\]
It follows from $\chi=1$ on $\Omega$ that $f_e=f$ almost everywhere
there, and $\supp f_e\subset\supp\chi$. Moreover, the extension satisfies
\begin{equation}\label{eq:variation-extension-size}
 |f_e(x)|\lesssim\mathbf{1}_{\supp\chi}(x)\|\mu\|_{\mathrm{TV}},
 \qquad
 \|f_e\|_{L^r(\mathbb{R}^d)}\lesssim\|\mu\|_{\mathrm{TV}},
 \quad0<r\leq\infty.
\end{equation}
For $|\alpha|=k+1$, we apply the distributional Leibniz rule to
$D^\alpha\chi_{\omega,b}$. The singular contribution is the measure
\[
 k!\,\chi(x)\omega^\alpha
 \delta_{\{\omega\cdot x=b\}},
\]
because $D^\alpha\sigma_k(\omega\cdot x-b)
=k!\omega^\alpha\delta_{\{\omega\cdot x=b\}}$ when $|\alpha|=k+1$.
Every other Leibniz term contains a derivative of $\chi$ and a derivative
of the ridge of order at most $k$, and is therefore a compactly supported
$L^1$ function. The $L^1$ norms of these terms and the total masses of
the hyperplane measures restricted to $\supp\chi$ are bounded uniformly
in $(\omega,b)$. Integrating against $\mu$ therefore gives
\begin{equation}\label{eq:variation-derivative-measures}
 \sum_{|\alpha|=m}\|D^\alpha f_e\|_{\mathrm{TV}}
 \lesssim\|\mu\|_{\mathrm{TV}}.
\end{equation}

We recall the annular multiplier $\varphi$ from the Littlewood--Paley
resolution in Section~2. There exist functions
$\rho_\alpha\in C_c^\infty(\mathbb{R}^d\setminus\{0\})$ such that
\begin{equation}\label{eq:multiplier-derivative-decomposition}
 \varphi(\xi)=\sum_{|\alpha|=k+1}\rho_\alpha(\xi)(\sqrt{-1}\xi)^\alpha.
\end{equation}
For example, since $\varphi$ vanishes near the origin, we may set
\[
 \rho_\alpha(\xi)
 :=\frac{\varphi(\xi)(-\sqrt{-1}\xi)^\alpha}
 {\sum_{|\beta|=k+1}\xi^{2\beta}},
 \qquad \xi\ne0,
\]
and set $\rho_\alpha(0)=0$. This is smooth because $\varphi$ vanishes on
a neighborhood of the origin.
For $j\geq1$, we set
$K_{\alpha,j}(x):=2^{jd}\rho^\vee_\alpha(2^jx)$. Then
\eqref{eq:multiplier-derivative-decomposition} gives
\begin{equation}\label{eq:variation-block-derivative-decomposition}
 2^{j(k+1)}\Delta_jf_e
 =\sum_{|\alpha|=k+1}K_{\alpha,j}*D^\alpha f_e,
 \qquad j\geq1.
\end{equation}

For a finite signed Radon measure $\nu$, we define
\[
 S_\alpha\nu(x)
 :=\left(\sum_{j\geq1}|(K_{\alpha,j}*\nu)(x)|^2\right)^{1/2}.
\]
We record the endpoint estimate used below. Plancherel's theorem and the
finite overlap of the supports of $\rho_\alpha(2^{-j}\cdot)$ give
\[
 \left\|\bigl(K_{\alpha,j}*u\bigr)_{j\geq1}
 \right\|_{L^2(\mathbb{R}^d;\ell^2)}
 \lesssim\|u\|_{L^2(\mathbb{R}^d)}.
\]
Moreover, the $\ell^2$-valued kernel
$\bm K_\alpha(x):=(K_{\alpha,j}(x))_{j\geq1}$ satisfies the standard
estimates
\[
 \|\bm K_\alpha(x)\|_{\ell^2}\lesssim |x|^{-d},
 \qquad
 \|\bm K_\alpha(x-y)-\bm K_\alpha(x)\|_{\ell^2}
 \lesssim\frac{|y|}{|x|^{d+1}}
\]
whenever $|x|>2|y|$. These estimates follow by splitting the dyadic sum
at the scale $2^j|x|\simeq1$ and using the Schwartz bounds for
$\rho_\alpha^\vee$ and its gradient. The Hilbert-valued
Calder\'on--Zygmund theorem
\cite{BenedekCalderonPanzone1962,RubioRuizTorrea1986,Stein1970} therefore
gives weak type $(1,1)$ for $L^1$ inputs.

The same weak-type estimate holds for finite measures. Indeed, we mollify
$\nu$ to obtain $\nu_\delta$, with
$\|\nu_\delta\|_{L^1}\leq\|\nu\|_{\mathrm{TV}}$. For each $j$,
$K_{\alpha,j}*\nu_\delta$ converges pointwise to
$K_{\alpha,j}*\nu$ as $\delta\to0$. Applying Fatou's lemma to finite
partial square functions and then letting the number of scales tend to
infinity gives
\begin{equation}\label{eq:measure-square-weak-one}
 \bigl|\{x\in\mathbb{R}^d:S_\alpha\nu(x)>\lambda\}\bigr|
 \lesssim\lambda^{-1}\|\nu\|_{\mathrm{TV}},
 \qquad\lambda>0.
\end{equation}
For the remainder of this estimate, assume that
$\supp\nu\subset\supp\chi$.
We choose a bounded open set $\mathcal O$ containing the closed unit
neighborhood of $\supp\chi$ and set $M:=\|\nu\|_{\mathrm{TV}}$.
By \eqref{eq:measure-square-weak-one},
\[
 |\{x\in\mathcal O:S_\alpha\nu(x)>\lambda\}|
 \lesssim\min\{|\mathcal O|,M/\lambda\}.
\]
Splitting the distribution-function integral at
$\lambda=M/|\mathcal O|$ gives, for $0<p<1$,
\begin{align*}
 \|S_\alpha\nu\|_{L^p(\mathcal O)}^p
 &=p\int_0^\infty\lambda^{p-1}
 \bigl|\{x\in\mathcal O:S_\alpha\nu(x)>\lambda\}\bigr|\,\dd\lambda\\
 &\lesssim |\mathcal O|^{1-p}M^p.
\end{align*}
Outside $\mathcal O$, Schwartz decay gives, for every $A>d/p$,
\[
 |(K_{\alpha,j}*\nu)(x)|
 \lesssim2^{-j(A-d)}(1+|x|)^{-A}M,
 \qquad x\in\mathcal O^c,\quad j\geq1.
\]
Taking the $\ell^2$ norm in $j$ and then the $L^p$ quasi-norm gives
\begin{equation*}
 \|S_\alpha\nu\|_{L^p(\mathcal O^c)}
 \lesssim M.
\end{equation*}
Consequently, every finite signed measure supported in $\supp\chi$
satisfies
\begin{equation}\label{eq:measure-square-p}
 \|S_\alpha\nu\|_{L^p(\mathbb{R}^d)}
 \lesssim\|\nu\|_{\mathrm{TV}},
 \qquad0<p<1.
\end{equation}

Combining \eqref{eq:variation-derivative-measures},
\eqref{eq:variation-block-derivative-decomposition}, and
\eqref{eq:measure-square-p}, and using $p$-subadditivity over the finitely
many multi-indices (all finite-dimensional $\ell^r$ quasi-norms are
equivalent), gives
\begin{equation*}
 \left\|\left(\sum_{j\geq1}
 |2^{jm}\Delta_jf_e|^2\right)^{1/2}\right\|_{L^p(\mathbb{R}^d)}
 \lesssim\|\mu\|_{\mathrm{TV}}.
\end{equation*}
For completeness, we set
$a_j(x):=|2^{jm}\Delta_jf_e(x)|^p$ and $r:=2/p>1$.
Minkowski's inequality in $\ell^r$ gives
\[
 \left\|\left(\int_{\mathbb{R}^d}a_j(x)\,\dd x\right)_{j\geq1}
 \right\|_{\ell^r}
 \leq\int_{\mathbb{R}^d}\|(a_j(x))_{j\geq1}\|_{\ell^r}\,\dd x.
\]
Taking the $p$-th root yields
\begin{align*}
 \left(\sum_{j=1}^\infty
 \bigl(2^{j(k+1)}\|\Delta_jf_e\|_{L^p(\mathbb{R}^d)}\bigr)^2\right)^{1/2}
 &\leq
 \left\|\left(\sum_{j=1}^\infty
 |2^{j(k+1)}\Delta_jf_e|^2\right)^{1/2}\right\|_{L^p(\mathbb{R}^d)}\\
 &\lesssim\|\mu\|_{\mathrm{TV}}.
\end{align*}
Finally, \eqref{eq:variation-extension-size}, the compact support of
$f_e$, and the Schwartz decay of $\varphi_0^\vee$ imply,
for every $B>0$,
\[
 |\Delta_0f_e(x)|
 \lesssim\|\mu\|_{\mathrm{TV}}(1+|x|)^{-B}.
\]
Choosing $Bp>d$ therefore gives
\[
 \|\Delta_0f_e\|_{L^p(\mathbb{R}^d)}\lesssim\|\mu\|_{\mathrm{TV}}.
\]
It follows that
\[
 \|f\|_{B_{p,2}^{k+1}(\Omega)}
 \leq\|f_e\|_{B_{p,2}^{k+1}(\mathbb{R}^d)}
 \lesssim\|f\|_{\mathcal L_1(\mathbb{D})}+\varepsilon.
\]
We let $\varepsilon\to0$ to prove the asserted embedding.

\subsection{Sharpness of the smoothness index}
The preceding argument proves the positive part of
Theorem~\ref{thm:variation-to-besov}. We first show that the exponent
$k+1$ cannot be increased by giving an explicit lower-bound calculation
for a single ridge. Since admissible
Littlewood--Paley resolutions give equivalent quasi-norms, we use the
radial resolution fixed in Section~2, for which
\[
 \varphi_j(\xi)=\varphi(2^{-j}\xi),\qquad j\geq1,
\]
where $\varphi$ is supported in an annulus and is nonzero on the first
coordinate axis. Writing
$x=(t,x')\in\mathbb{R}\times\mathbb{R}^{d-1}$, define the one-dimensional
kernel and profile by
\begin{equation*}
 Q(t):=\int_{\mathbb{R}^{d-1}}\varphi^\vee(t,x')\,\dd x',
 \qquad
 G_k(t):=\int_{\mathbb{R}}Q(v)\sigma_k(t-v)\,\dd v.
\end{equation*}
The Fourier transform of this profile is
\begin{equation*}
 \widehat{G_k}(\xi)
 =c_k\varphi(\xi,0)(\sqrt{-1}\xi)^{-k-1},
 \qquad \xi\neq0,
\end{equation*}
so $G_k\in\mathcal S(\mathbb{R})$ and $G_k\not\equiv0$. After rotating
coordinates so that $\omega=e_1$, direct rescaling of the convolution
gives the exact identity
\begin{equation*}
 \Delta_j[\sigma_k(x_1-b)](t,x')
 =2^{-jk}G_k\bigl(2^j(t-b)\bigr),
 \qquad j\geq1.
\end{equation*}

We choose a bounded measurable set $E\subset\mathbb{R}$ of positive measure
and $c_0>0$ such that $|G_k(t)|\geq2c_0$ on $E$, and fix $R>0$ such that
$E\subset[-R,R]$.
We choose open intervals $I_0\Subset I\Subset(-c,c)$, a point
$x'_0\in\mathbb{R}^{d-1}$, and $r_0>0$. We set
\[
 U:=B_{\mathbb{R}^{d-1}}(x'_0,2r_0),
 \qquad U_0:=B_{\mathbb{R}^{d-1}}(x'_0,r_0),
\]
and choose these objects so that $I\times U\Subset\Omega$. We choose
$\rho\in C_c^\infty(\Omega)$ with $\rho=1$ on a neighborhood of
$\overline{I\times U}$. When $d=1$, the transverse variables are
omitted and the measure of $U_0$ below is interpreted as one. For
$b\in I_0$ and $(t,x')\in\mathbb{R}^d$, let
\begin{align*}
 \mathcal E_{j,b}(t,x')
 :=2^{-jk}G_k\bigl(2^j(t-b)\bigr)-\Delta_j[\rho\sigma_k(x_1-b)](t,x').
\end{align*}
Then, for every $B>0$, we have
\begin{align}
 \sup_{\substack{b,t\in I_0,\ |u|\leq R\\x'\in U_0}}
 \bigl|\mathcal E_{j,b}(t+2^{-j}u,x')\bigr|
 &\lesssim2^{-jB}.
 \label{eq:ridge-profile-remainder}
\end{align}
Indeed, the points $(t+2^{-j}u,x')$ remain in a fixed compact subset of
$I\times U$ for all sufficiently large $j$, uniformly in
$t\in I_0$, $|u|\leq R$, and $x'\in U_0$. The distance from this
compact set to the support of $1-\rho$ is positive, and
$(1-\rho(x))\sigma_k(x_1-b)$ vanishes on a fixed neighborhood of those
points and grows at most polynomially of degree $k$, uniformly for
$b\in I_0$. Applying the convolution formula for $\Delta_j$, changing
variables $z=2^j(x-y)$, and using a Schwartz bound of order greater than
$B+d+k$ proves \eqref{eq:ridge-profile-remainder}.
Taking $B>k$ in
\eqref{eq:ridge-profile-remainder}, we obtain for all sufficiently large
$j$,
\[
 |\Delta_j[\rho\sigma_k(x_1-b)](b+2^{-j}u,x')|
 \gtrsim 2^{-jk},\qquad (u,x')\in E\times U_0,
\]
uniformly for $b\in I_0$. For $s>0$ and $0<q\leq\infty$, we use the
localization inequality
\begin{equation}\label{eq:besov-interior-localization}
 \|\rho W\|_{B^s_{p,q}(\mathbb{R}^d)}
 \lesssim\|W\|_{B^s_{p,q}(\Omega)}.
\end{equation}
Indeed, we define $\rho W$ to be zero outside $\Omega$. If $W_e$ is any
extension of $W$, then $\rho W=\rho W_e$ on $\mathbb{R}^d$, and
multiplication by a smooth compactly supported function is bounded on
$B^s_{p,q}(\mathbb{R}^d)$ \cite{Triebel,FrazierJawerth1990}. Taking
the infimum over $W_e$ proves \eqref{eq:besov-interior-localization}.
For every $N\geq2$, we choose equally spaced points
$t_{1,N},\ldots,t_{N,N}$ in $I_0$ and set
\begin{equation*}
 V_N(t,x'):=\frac1N\sum_{i=1}^N\sigma_k(t-t_{i,N}).
\end{equation*}
It follows directly from the atomic representation that
$\|V_N\|_{\mathcal L_1(\mathbb{D})}\leq1$. We write
$d_N:=t_{i+1,N}-t_{i,N}\simeq N^{-1}$. We fix $M>1$ and choose
$\delta>0$, independently of $N$, sufficiently small for the estimates
below. Choose $j_N$ such that
\begin{equation}\label{eq:many-ridges-scale}
 \delta d_N\leq2^{-j_N}<2\delta d_N.
\end{equation}
This choice of $\delta$ makes the slabs
\begin{equation*}
 \mathcal C_{i,N}:=
 \{(t_{i,N}+2^{-j_N}u,x'):u\in E,\ x'\in U_0\},
 \qquad1\leq i\leq N,
\end{equation*}
pairwise disjoint and ensures that, for $u\in E$ and $r\ne i$,
\[
 |u+2^{j_N}(t_{i,N}-t_{r,N})|
 \gtrsim\delta^{-1}|i-r|.
\]
On $\mathcal C_{i,N}$, equation~\eqref{eq:ridge-profile-remainder} gives
\begin{equation}\label{eq:many-ridges-profile}
\begin{aligned}
 \Delta_{j_N}(\rho V_N)
 (t_{i,N}+2^{-j_N}u,x')
 ={}&\frac{2^{-j_Nk}}{N}
 \sum_{r=1}^NG_k\bigl(u+2^{j_N}(t_{i,N}-t_{r,N})\bigr)\\
 &-\mathcal E_{i,N}(u,x').
\end{aligned}    
\end{equation}

Here the averaged cutoff error is
\[
 \mathcal E_{i,N}(u,x')
 :=\frac1N\sum_{r=1}^N
 \mathcal E_{j_N,t_{r,N}}
 \bigl(t_{i,N}+2^{-j_N}u,x'\bigr).
\]
Since $G_k\in\mathcal S(\mathbb{R})$, for every $M>1$,
\begin{align}
 \sup_{u\in E}\sum_{r\neq i}
 |G_k(u+2^{j_N}(t_{i,N}-t_{r,N}))|
 &\lesssim\sum_{n\geq1}(1+\delta^{-1}n)^{-M}
 \lesssim\delta^M.
 \label{eq:many-ridges-tail}
\end{align}
By the initial choice of $\delta$, the last quantity is at most
$c_0/2$. Fix also $B>k+1$. Since
$N\simeq2^{j_N}$ by \eqref{eq:many-ridges-scale}, the averaged cutoff
remainder satisfies
\[
 |\mathcal E_{i,N}(u,x')|
 \lesssim2^{-j_NB},\qquad
 |\mathcal E_{i,N}(u,x')|
 \leq\frac{c_0}{2N}2^{-j_Nk},
\]
for all sufficiently large $N$ and uniformly for $(u,x')\in E\times U_0$.
Combining \eqref{eq:many-ridges-profile} and
\eqref{eq:many-ridges-tail} gives
\begin{equation*}
 |\Delta_{j_N}(\rho V_N)(x)|
 \gtrsim N^{-1}2^{-j_Nk},\qquad x\in\mathcal C_{i,N}.
\end{equation*}
Since $|\mathcal C_{i,N}|=2^{-j_N}|E||U_0|$, summing over the disjoint
slabs gives
\begin{align*}
 \|\Delta_{j_N}(\rho V_N)\|_{L^p(\mathbb{R}^d)}^p
 &\gtrsim N^{1-p}2^{-j_N(kp+1)},
 \\
 \|\Delta_{j_N}(\rho V_N)\|_{L^p(\mathbb{R}^d)}
 &\gtrsim N^{1/p-1}2^{-j_N(k+1/p)}
 \gtrsim2^{-j_N(k+1)}.
\end{align*}
The last inequality follows from \eqref{eq:many-ridges-scale} and
$d_N\simeq N^{-1}$. Hence, by \eqref{eq:besov-interior-localization},
for every $s>k+1$ and every $0<q\leq\infty$,
\begin{align*}
 \|V_N\|_{B_{p,q}^s(\Omega)}
 &\gtrsim\|\rho V_N\|_{B_{p,q}^s(\mathbb{R}^d)}
 \gtrsim2^{j_Ns}
 \|\Delta_{j_N}(\rho V_N)\|_{L^p}\\
 &\gtrsim2^{j_N(s-k-1)}\longrightarrow\infty.
\end{align*}
Since $\|V_N\|_{\mathcal L_1(\mathbb{D})}\leq1$, the exponent $k+1$ is the
largest possible smoothness exponent for $0<p<1$.

\subsection{Sharpness of the fine index}
The separated-ridge construction rules out every smoothness exponent above
$k+1$, but it does not distinguish the fine index at the critical
smoothness. We therefore use a lacunary superposition to show that the fine
index $2$ cannot be decreased. Retain the
radial resolution and the sets $I_0\Subset I$ and $U_0\Subset U$ chosen
above, as well as $m=k+1$. Choose a nonzero real-valued function
$\theta\in C_c^\infty(\mathbb{R})$ with $\supp\theta\subset I_0$, a compact
interval $K_0\Subset I$ whose interior contains $\supp\theta$, and a
number $\tau_0>0$ such that $\varphi(\tau_0,0)\neq0$. We define
\[
 a_m(\eta):=k!\varphi(\eta,0)(\sqrt{-1}\eta)^{-m},
 \qquad \eta\neq0,
\]
and extend it by zero near the origin. Then
$a_m\in C_c^\infty(\mathbb{R}\setminus\{0\})$ and
$a_m(\tau_0)\neq0$. We choose an integer $L$ so large that, whenever
$r\ne\ell$, the points $\pm\tau_0 2^{L(r-\ell)}$ lie outside
$\supp a_m$ with a fixed positive relative separation.

Given $N\geq1$, we choose an integer $J_*=J_*(N)$, to be fixed below, and
set
\[
 J_\ell:=J_*+L\ell,
 \qquad \lambda_\ell:=\tau_0 2^{J_\ell},
 \qquad1\leq\ell\leq N.
\]
We define the lacunary density by
\begin{equation*}
 g_N(b):=\frac1{\sqrt N}\theta(b)
 \sum_{\ell=1}^N\cos(\lambda_\ell b).
\end{equation*}
Expanding the square, the off-diagonal integrals are controlled by values
of $\widehat{\theta^2}$ at the frequencies
$\pm\lambda_\ell\pm\lambda_r$. Lacunarity and the Schwartz decay of
$\widehat{\theta^2}$ show more precisely that, for every $A>0$,
\[
 \sum_{1\leq r<\ell\leq N}
 \big|\widehat{\theta^2}(\lambda_\ell\pm\lambda_r)\big|
 \lesssim_A\sum_{\ell\geq1}\ell
 (1+\tau_02^{J_*+L\ell})^{-A}.
\]
Thus the sum of the off-diagonal contributions is bounded uniformly in
$N$ and $J_*\geq J_0$, whereas the sum of the $N$ diagonal contributions
is $O(N)$. Accounting for the factor $N^{-1}$ from the normalization in
the definition of $g_N$, we obtain
\[
 \|g_N\|_{L^2(\mathbb{R})}^2
 \lesssim1.
\]
The Cauchy--Schwarz inequality therefore gives
\begin{equation}\label{eq:lacunary-density-L1}
 \|g_N\|_{L^1(\mathbb{R})}
 \leq |I_0|^{1/2}\|g_N\|_{L^2(\mathbb{R})}\lesssim1.
\end{equation}
We define the corresponding ridge function by
\begin{equation*}
 H_N(t,x'):=\int_{I_0}\sigma_k(t-b)g_N(b)\,\dd b.
\end{equation*}
This is an admissible ridge integral with direction $e_1$, so
\eqref{eq:lacunary-density-L1} implies
\begin{equation}\label{eq:lacunary-variation-bound}
 \|H_N\|_{\mathcal L_1(\mathbb{D})}\lesssim1.
\end{equation}
Differentiation in the distributional sense gives
\begin{equation*}
 \partial_t^mH_N(t,x')=k!g_N(t).
\end{equation*}

For an integer $J$, let $T_J$ denote the one-dimensional multiplier
\[
 T_Ju:=\bigl(a_m(2^{-J}\cdot)\widehat u\bigr)^\vee.
\]
Since annular multipliers annihilate the distribution supported at the
origin that may arise when solving
$(\sqrt{-1}\xi)^m\widehat H_N=k!\widehat g_N$, we have the exact identity
\[
 2^{Jm}\Delta_JH_N(t,x')=T_Jg_N(t),
 \qquad J\geq1.
\]

We next isolate the resonant frequency. Modulation and a change of
variables in Fourier space give, as $J\to\infty$,
\begin{align*}
 T_J(\theta e^{\sqrt{-1}\tau_02^J\cdot})(t)
 &=e^{\sqrt{-1}\tau_02^Jt}
 \bigl(a_m(\tau_0)\theta(t)+o_{\mathcal S}(1)\bigr),\\
 T_J(\theta e^{-\sqrt{-1}\tau_02^J\cdot})(t)
 &=e^{-\sqrt{-1}\tau_02^Jt}
 \bigl(a_m(-\tau_0)\theta(t)+o_{\mathcal S}(1)\bigr).
\end{align*}
Since $a_m(-\tau_0)=\overline{a_m(\tau_0)}$, periodic averaging of
$|\cos|^p$ on the set where $\theta$ is nonzero gives a constant $c_1>0$
such that, for all sufficiently large $J$,
\begin{equation}\label{eq:lacunary-resonant-lower}
 \|T_J(\theta\cos(\tau_02^J\cdot))\|_{L^p(K_0)}\geq c_1.
\end{equation}

The other frequencies are uniformly negligible. Indeed, Fourier
separation and the Schwartz decay of $\widehat\theta$ imply that, for
every $M>0$,
\[
 \|T_{J_\ell}(\theta e^{\pm\sqrt{-1}\lambda_r\cdot})\|_{L^p(K_0)}
 \lesssim
 \begin{cases}
  2^{-MJ_\ell},&r<\ell,\\
  2^{-MJ_r},&r>\ell.
 \end{cases}
\]
To see this, on the support of
$a_m(2^{-J_\ell}\xi)$ the shifted variable
$\xi\mp\lambda_r$ has size $\gtrsim2^{J_\ell}$ when $r<\ell$ and
$\gtrsim2^{J_r}$ when $r>\ell$. The domain of integration in the Fourier
integral has length $O(2^{J_\ell})$. Hence Schwartz decay of order
$R>M+1$ gives the stated $L^\infty$ bounds (for $r>\ell$, use
$2^{J_\ell}(1+2^{J_r})^{-R}\lesssim2^{-MJ_r}$). The displayed
$L^p$ estimate now follows from
$\|u\|_{L^p(K_0)}\leq |K_0|^{1/p}\|u\|_{L^\infty(K_0)}$.
In particular,
\[
 \sum_{r\ne\ell}
 \|T_{J_\ell}(\theta e^{\pm\sqrt{-1}\lambda_r\cdot})\|_{L^p(K_0)}^p
 \leq C_M\left(\ell2^{-MpJ_\ell}
       +\sum_{r>\ell}2^{-MpJ_r}\right)
 \leq C'_M2^{-MpJ_*}.
\]
Choosing $J_*$ sufficiently large, with the lower bound independent of
$N$ and $\ell$, and using the rearranged $p$-subadditivity
\[
 \|u+v\|_{L^p}^p\geq\|u\|_{L^p}^p-\|v\|_{L^p}^p,
\]
together with \eqref{eq:lacunary-resonant-lower}, we obtain
\begin{equation}\label{eq:lacunary-unlocalized-block}
 \|T_{J_\ell}g_N\|_{L^p(K_0)}\gtrsim N^{-1/2},
 \qquad1\leq\ell\leq N.
\end{equation}

It remains to insert the interior cutoff. From
\eqref{eq:lacunary-density-L1},
\[
 |H_N(t,x')|\lesssim(1+|t|)^k
\]
uniformly in $N$. The function $(1-\rho)H_N$ vanishes in a fixed
neighborhood of $K_0\times U_0$ and has at most polynomial growth.
The separated-support Schwartz-tail estimate used in
\eqref{eq:ridge-profile-remainder} therefore gives, for every $B>0$,
\[
 \sup_{x\in K_0\times U_0}
 |\Delta_J[(1-\rho)H_N](x)|\lesssim2^{-JB}.
\]
We choose $B>m$ and enlarge $J_*(N)$ until
\[
 C2^{-J_1(B-m)}\leq \varepsilon_0 N^{-1/2},
 \qquad J_1=J_*+L,
\]
with $\varepsilon_0>0$ a sufficiently small fixed constant. Then, simultaneously
for every $\ell$, the cutoff error multiplied by $2^{J_\ell m}$ has
$L^p(K_0\times U_0)$ quasi-norm at most $\varepsilon_0N^{-1/2}$. Combining this
estimate with \eqref{eq:lacunary-unlocalized-block} and using
$p$-subadditivity gives
\begin{equation*}
 2^{J_\ell m}\|\Delta_{J_\ell}(\rho H_N)\|_{L^p(\mathbb{R}^d)}
 \gtrsim N^{-1/2},
 \qquad1\leq\ell\leq N.
\end{equation*}
Consequently, for every $0<q<2$, the localization inequality
\eqref{eq:besov-interior-localization} gives
\begin{align*}
 \|H_N\|_{B_{p,q}^{k+1}(\Omega)}
 &\gtrsim\|\rho H_N\|_{B_{p,q}^{k+1}(\mathbb{R}^d)}\\
 &\gtrsim
 \left(\sum_{\ell=1}^N N^{-q/2}\right)^{1/q}
 =N^{1/q-1/2}\longrightarrow\infty.
\end{align*}
Together with \eqref{eq:lacunary-variation-bound}, this proves
\[
 \mathcal L_1(\mathbb{D})\not\hookrightarrow B_{p,q}^{k+1}(\Omega),
 \qquad0<q<2,
\]
and completes the proof.
\end{proof}

\section{Conclusion}\label{sec:conclusion}

We have established embeddings in both directions between isotropic
quasi-Banach Besov spaces and the variation space generated by a normalized
shallow $\operatorname{ReLU}^k$ ridge dictionary. For $0<p\leq1$, the
forward embedding holds above the smoothness threshold $k+d/p$ for every
$0<q\leq\infty$, and at the threshold when $0<q\leq1$. A rescaled-bump
family shows that the smoothness threshold cannot be lowered. For
$0<p<1$, we proved
$\mathcal L_1(\mathbb{D})\hookrightarrow B_{p,2}^{k+1}(\Omega)$. Separated-ridge
and lacunary-ridge families show, respectively, that the smoothness
$k+1$ and the fine index $2$ cannot be improved. Combining the forward
embedding with \eqref{eq:SiegelXu} also gives the stated $L^r$ best
$n$-term estimate for every $2\leq r<\infty$.

The sharp isotropic theory established here opens several concrete
directions. Natural next problems are to determine the embedding
thresholds for anisotropic and mixed-smoothness scales and for other
activation functions, and to combine the present representation estimates
with PDE-specific stability bounds for end-to-end discretization and
training analysis.

\end{document}